\documentclass[a4paper]{amsart}

\usepackage[utf8]{inputenc}
\usepackage{lmodern}
\usepackage{array}
\usepackage{graphicx}
\usepackage{stmaryrd}
\usepackage{subfig}
\usepackage{amssymb, amsfonts, amsthm, amsmath}
\usepackage{enumerate}
\usepackage{enumitem}
\usepackage{bbm}
\usepackage[english]{babel}
\usepackage{tikz}
\usepackage{comment}
\usepackage[margin=0.8in]{geometry}
\usetikzlibrary{decorations}
\usetikzlibrary{decorations.pathreplacing}
\tikzset{individu/.style={draw,thick}}

\numberwithin{equation}{section}

\theoremstyle{plain}
\newtheorem{theorem}{Theorem}[section]
\newtheorem{corollary}[theorem]{Corollary}

\newtheorem{lemma}[theorem]{Lemma}
\newtheorem{proposition}[theorem]{Proposition}

\theoremstyle{definition}
\newtheorem{definition}[theorem]{Definition}

\theoremstyle{remark}
\newtheorem{remark}[theorem]{Remark}
\newtheorem{example}[theorem]{Example}

\definecolor{calpolypomonagreen}{rgb}{0, 0.6, 0.2}

\newcommand{\e}{\mathbf{e}}
\renewcommand{\v}{\mathbf{v}}
\newcommand{\vtilde}{\tilde{\mathbf{v}}}
\newcommand{\Rvec}{\mathbf{R}}

\renewcommand{\tilde}[1]{\widetilde{#1}}
\renewcommand{\epsilon}{\varepsilon}
\renewcommand{\phi}{\varphi}

\renewcommand{\e}{\mathrm{e}}

\newcommand{\wt}{\mathrm{wt}}
\newcommand{\row}{\mathrm{row}}

\newcommand{\sgn}{\mathrm{sgn}}

\title{Vector-relation configurations and plabic graphs}

\author{Niklas Affolter}
\address{Technische Universit\"at Berlin, Institute of Mathematics, Strasse des 17. Juni 136,
	10623 Berlin, Germany}
\author{Max Glick}
\address{Google Inc, Pittsburgh, PA 15206, USA}
\author{Pavlo Pylyavskyy}
\address{Department of Mathematics, University of Minnesota, Minneapolis, MN 55455, USA}
\author{Sanjay Ramassamy}
\address{Universit\'e Paris-Saclay, CNRS, CEA, Institut de Physique Th\'eorique, 91191 Gif-sur-Yvette, France}
\thanks{N. A. was supported by the Deutsche Forschungsgemeinschaft (DFG) Collaborative Research
Center TRR 109 ``Discretization in Geometry and Dynamics'' and by the ENS-MHI chair funded
by MHI. N. A. and S. R. were partially supported by the Agence Nationale de la Recherche, Grant
Number ANR-18-CE40-0033 (ANR DIMERS). P. P. was partially supported by NSF grants DMS-1148634 and DMS-1351590. S. R. was also partially supported by the CNRS grant Tremplin@INP, which funded a visit of N. A. to Paris-Saclay, as well as by the Fondation Sciences Math\'ematiques de Paris.}

\begin{document}

\begin{abstract}
We study a simple geometric model for local transformations of bipartite graphs.  The state consists of a choice of a vector at each white vertex made in such a way that the vectors neighboring each black vertex satisfy a linear relation.  The evolution for different choices of the graph coincides with many notable dynamical systems including the pentagram map, $Q$-nets, and discrete Darboux maps.  On the other hand, for plabic graphs we prove unique extendability of a configuration from the boundary to the interior, an elegant illustration of the fact that Postnikov's boundary measurement map is invertible.  In all cases there is a cluster algebra operating in the background, resolving the open question for $Q$-nets of whether such a structure exists.
\end{abstract}

\maketitle

\section{Introduction} \label{sec:intro}

The dynamics of local transformations on weighted networks play a central role in a number of settings within algebra, combinatorics, and mathematical physics.  In the context of the dimer model on a torus, these local moves give rise to the discrete cluster integrable systems of Goncharov and Kenyon \cite{GK}.  Meanwhile, for plabic graphs in a disk, Postnikov transformations relate different parametrizations of positroid cells \cite{Pos06} which in turn define a stratification of the totally non-negative Grassmannian.

The dimer model also manifests itself in many geometrically defined dynamical systems.  We focus on projective geometry and draw our initial motivation from the pentagram map.  The pentagram map was defined by Schwartz \cite{S92} and related in \cite{G11} to coefficient-type cluster algebra dynamics \cite{FZ}.  Gekhtman, Shapiro, Tabachnikov, and Vainshtein  \cite{GSTV12,GSTV16} placed the pentagram map and certain generalizations in the context of weighted networks and derived a more conceptual take on the integrability property first proven by Ovsienko, Schwartz, and Tabachnikov \cite{OST1}.  Although considerable work in various directions of the subject has been undertaken, most relevant to our work is a further generalization termed $Y$-meshes \cite{GP16}.

We propose a simple but versatile geometric model for the space of edge weights of any bipartite graph modulo gauge equivalence, with applications to the fields of both geometric dynamics and plabic graphs. The induced dynamics of local transformations provides an analog of the pentagram map for every planar bipartite graph and includes as special cases generalized pentagram maps, $Q$-nets, and discrete Darboux maps.  This common generalization resolves a long standing question \cite[Remark 1.5]{GP16} of how the pentagram map and $Q$-nets relate.  Moreover, our systems come with cluster dynamics, which is new in the $Q$-net case and should be of interest to discrete differential geometers.  Lastly, in the setting of plabic graphs we define a geometric version of the boundary measurement map and its inverse.  In this language, properties of the boundary measurement map imply the unique solvability of a certain family of geometric realization problems. The geometric  model story runs parallel to the classical one of planar weighted bipartite graphs, with the concepts of gauge transformations, local transformations and face variables of the latter bearing simple geometric interpretations (see Sections \ref{sec:transform} and \ref{sec:face}) in the former.

\subsection{Overview of main definitions and results}
Our main object of study is a certain collection of geometric data, which we term a \emph{vector-relation configuration}, associated to a bipartite graph.  Roughly speaking, such a configuration consists of a choice of vector (from some fixed vector space) associated to each white vertex of the graph, with the property that the set of vectors neighboring each black vertex satisfy a linear relation.  The exact requirements vary a bit based on the context and are described in Definitions \ref{def:vr} and \ref{def:vr-plabic}.

In the case of a planar bipartite graph, we additionally define evolution equations of vector-relation configurations under local transformations.  In parallel with the dynamics of edge-weighted graphs, these operations preserve a notion of gauge equivalence.  In fact, we will show (Proposition \ref{prop:moves}) that these two stories are in some sense equivalent to each other.  At least locally, it is possible to go back and forth between edge weights and vector-relation configurations (with some genericity assumptions) in a manner that commutes with local transformations.  As a result, we can import much of the theory of the dimer model to our setting.  For instance we get face weights, which are simple to define geometrically in terms of multi-ratios (Proposition \ref{prop:faceweights}) and which satisfy simple, rational evolution equations.

For roughly the second half of the paper, we focus our attention on plabic graphs in a disk.  We assume all boundary vertices are white, meaning that a vector-relation configuration on such a graph includes a vector at each boundary vertex.  Although local transformations are also of interest in this case, we focus on global questions concerning the space of all configurations given fixed $G$.  The main result, which in isolation is rather striking, is that a configuration is uniquely determined up to gauge by its boundary vectors.

To state this result more precisely and give relevant context, we recall that each plabic graph gives rise to a combinatorial object called a positroid and a geometric object called a positroid variety.  Let $G$ be a plabic graph.  We will let $\mathcal{M}$ denote the associated positroid and $\Pi_{\mathcal{M}}$ the associated positroid variety.  A fundamental object in this area is the boundary measurement map which takes as input an edge-weighting of $G$ and outputs a point of $\Pi_{\mathcal{M}}$.

\begin{theorem} \label{thm:plabic}
Fix a plabic graph $G$.
\begin{enumerate}
\item Given a vector-relation configuration on $G$, the matrix $A = [v_1 \cdots v_n]$ whose columns are the boundary vectors of the configuration lies in the positroid variety $\Pi_{\mathcal{M}}$.
\item Suppose $G$ is reduced.  There is a dense subset $T_G \subseteq \Pi_{\mathcal{M}}$ such that for $A \in T_G$, the columns $v_1,\ldots, v_n$ of $A$ can be extended to a vector-relation configuration on $G$ that is unique up to gauge at internal vertices.  In particular, each internal vector is determined up to scale.
\end{enumerate}
\end{theorem}

The definition of a vector-relation configuration on a plabic graph is given in Definition~\ref{def:vr-plabic}. We review the definition of reducedness for plabic graphs in Section \ref{sec:background}, which contains background on various aspects of positroid theory.  Also, note that we mostly assume boundary vertices in plabic graphs have degree $1$, but in certain examples such as the following it is convenient to allow larger degree.  Our main results can be generalized to this situation, but it makes some definitions and arguments more cumbersome.

\begin{example} \label{ex:gr36}
Consider the plabic graph $G$ in Figure \ref{fig:gr36}.  The associated positroid variety is the full Grassmannian $Gr_{3,6}$.  As such, Theorem \ref{thm:plabic} asserts that the boundary vectors $v_1,\ldots v_6 \in \mathbb{C}^3$ of a configuration can be chosen generically and the last vector $u$ is determined by them up to scale.

Indeed suppose $v_1,\ldots v_6 \in \mathbb{C}^3$ are given and consider the possibilities for the internal vector $u$.  The lower black vertex forces $u,v_1,v_2$ to be dependent while the top black vertex forces $u,v_4,v_5$ to be dependent.  If the $v_i$ are generic then $u$ must lie on the line of intersection of the planes $\langle v_1,v_2 \rangle$ and $\langle v_4, v_5 \rangle$.  Hence $u$ is determined up to scale.  The other two black vertices have degree $4$.  It is always possible to find a linear relation among $4$ vectors in $\mathbb{C}^3$, so there are no added conditions imposed on $u$.
\end{example}

\begin{figure}
\centering
\includegraphics[height=2.5in]{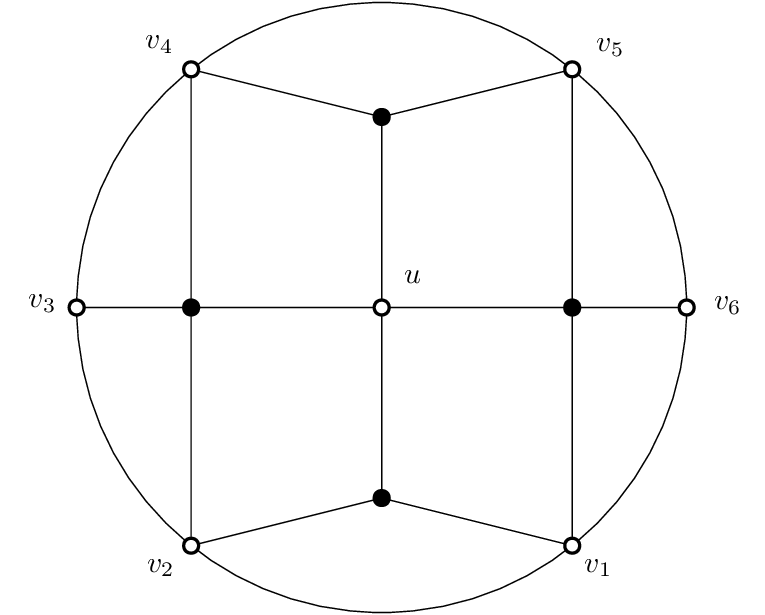}
\caption{A plabic graph corresponding to the open cell in $Gr(3,6)$}
\label{fig:gr36}
\end{figure}

\subsection{Relation to previous work} Our model of vector-relation configurations has substantial precedent in the literature.  In fact, a main selling point of our specific formulation is that it is versatile enough to tie into previously studied ideas in a variety of areas.  We outline some of the relevant previous work here for the interested reader's convenience.

In the plabic graph setting, Lam's \emph{relation space} \cite[Section 14]{LamCDM} is in a sense dual to our model.  
Let $G$ be a plabic graph and suppose we have vectors $v_w \in \mathbb{C}^k$ at white vertices satisfying relations
\begin{displaymath}
\sum_{w} K_{bw} v_w = 0
\end{displaymath}
indexed by black vertices.  Our approach is to consider the boundary vectors $[v_1 \cdots v_n]$ as making up a point in $Gr_{k,n}$.  The relation space is the dual point of $Gr_{n-k,n}$, that is, the kernel of $[v_1 \cdots v_n]$.  More directly, one takes in the $n$-dimensional space of linear combinations of $v_1,\ldots, v_n$ the subspace consisting of valid relations.  Note that the coefficients $K_{bw}$ alone determine the relation space, so the $v_w$ are replaced with formal variables.  In light of this connection, our Proposition \ref{prop:plabicSigns} is equivalent to \cite[Theorem 14.6]{LamCDM} except that we give explicit rules for the signs.

Another geometric model on plabic graphs is provided by Postnikov \cite{Pos18}.  He associates a point of a small Grassmannian $Gr_{1,3}$ or $Gr_{2,3}$ to each vertex.  His setup has the advantage that there is a natural duality between the black and white vertices.  We should also note that both \cite[Section 14]{LamCDM} and \cite{Pos18} are attempts to put on more mathematical footing the on-shell diagrams of physics \cite{Nima}.

It should be no surprise to experts that vector-relation configurations on plabic graphs have a close connection to the boundary measurement map, see Section \ref{sec:measurement}.  Taking this connection as given, Theorem \ref{thm:plabic} can be derived from corresponding properties of the boundary measurement map, the most difficult of which were proven by Muller and Speyer \cite{MulSpe}.  We take a different path, proving Theorem \ref{thm:plabic} directly to highlight some of the strengths of our model.  For instance, the analog for us of the inverse of the boundary measurement map is a novel reconstruction map which has a very pleasant geometric description.  This all said, we do make extensive use of a number of combinatorial and geometric results that are proven in the earlier sections of \cite{MulSpe}.

In the case of the dimer model on the torus, Kenyon and Okounkov \cite{KO} associate a section of a certain line bundle to each white vertex of a bipartite graph.  It is easy to see that said sections satisfy linear relations in such a way as to give a configuration (in an infinite dimensional space).  Fock \cite{F15} shows how to recover this data from the line bundle.  He constructs on each vertex of one color (black with his conventions) a one dimensional space defined by a certain intersection of spaces living on zigzags.  Our reconstruction map for plabic graphs as defined by \eqref{eq:reconstruct} is entirely analogous.  

As already mentioned, Gekhtman et al. \cite{GSTV12, GSTV16} were the first to describe the pentagram map (and generalizations) in terms of dynamics on networks.  It is easy in retrospect to see all of the ideas of vector-relation configurations in these papers.  For instance, the authors identify the edge weights as coefficients of linear relations among lifts of the points of the polygon.  Such coefficients also appear as the $a,b$-coordinates of Ovsienko, Schwartz, and Tabachnikov \cite{OST1}.  Similarly, in the study of $Q$-nets \cite{BS} an important role is played by the relation among the four coplanar points living at the vertices of each primitive square.

Finally, we note that there are many other geometric models on planar bipartite graphs compatible with the dimer model on the torus, for instance $T$-graphs \cite{KenShe}, Miquel dynamics on circle patterns \cite{A18,KLRR}, and Clifford dynamics \cite{KonSch}.  The interplay between the various models is considered in \cite{AGR}.  That paper also includes descriptions of both $Q$-nets and discrete Darboux maps in terms of cluster dynamics which differ from those in the present paper.

\subsection{Structure of the paper} The remainder of this paper is organized as follows.  We begin in Section \ref{sec:transform} by reviewing the dynamics of local transformations and providing the main definitions for vector-relation configurations.  Section \ref{sec:config} covers the basic properties of our vector-relation model as well as a slight modification with the ambient vector space replaced by its projectivization.  In Section \ref{sec:examples} we illustrate how to incorporate several previously studied systems into our framework.  In Section \ref{sec:resist} we identify what sorts of vector-relation configurations arise from resistor and Ising networks.  We tackle the plabic graph case in Section \ref{sec:plabic}, building the general theory and proving Theorem \ref{thm:plabic}.  We relate our model with the boundary measurement map in Section \ref{sec:measurement}.  Finally, Section \ref{sec:structure} examines the geometry of the space of configurations on a plabic graph.

\medskip

\textbf{Acknowledgments.} We thank Lie Fu, Rick Kenyon, and Kelli Talaska for many helpful conversations.  We thank the anonymous referee for extensive comments and suggestions that led to several improvements to this paper.

\section{Background and main definitions} \label{sec:transform}

We first recall the classical setting of weighted bipartite  planar graphs and their transformations, before introducing our geometric model of vector-relation configurations on bipartite planar graphs and the corresponding transformations on such configurations.

Let $G$ be a planar bipartite graph with nonzero edge weights.  A \emph{gauge transformation} at a given vertex multiplies the weights of all edges incident to that vertex by a common scalar.  A \emph{local transformation} modifies a small portion of $G$ in the manner indicated in one of the pictures in Figure \ref{fig:moves}.  There are two types of local transformations:
\begin{itemize}
\item The top of Figure \ref{fig:moves} depicts \emph{urban renewal}.  The new edge weights are
\begin{equation} \label{eq:urban}
a' = \frac{a}{ac+bd}, \quad b' = \frac{b}{ac+bd}, \quad c' = \frac{c}{ac+bd}, \quad d' = \frac{d}{ac+bd}.
\end{equation}
This transformation is only defined if $ac+bd \neq 0$.
\item The bottom of Figure \ref{fig:moves} depicts \emph{degree two vertex addition}.  A vertex is split into two vertices of the same color connected by a new degree two vertex of the opposite color.  The move depends on a choice of a partition of the neighbors of the original vertex into two cyclically consecutive blocks of size $k$ and $l$.  The figure depicts addition of a degree two black vertex, but the same move is allowed with all colors reversed producing a degree two white vertex instead.
\end{itemize}

\begin{figure}
\centering
\includegraphics[height=3in]{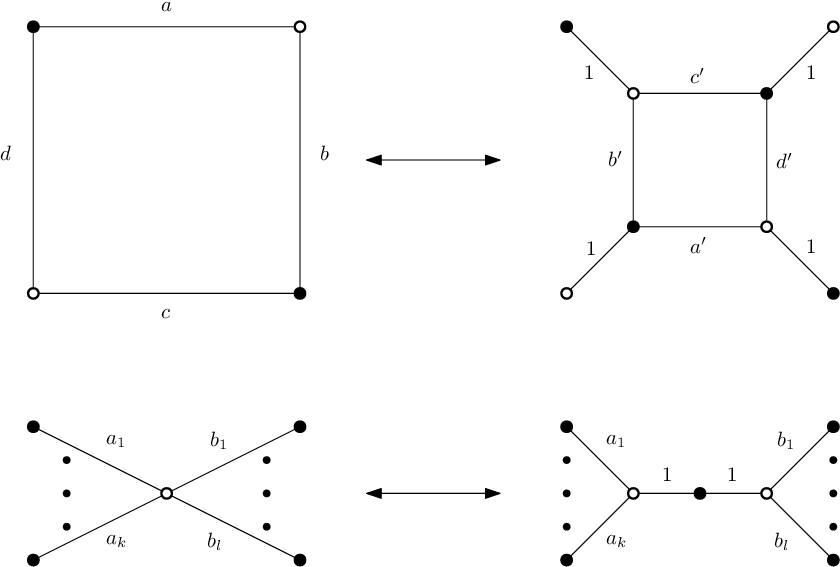}
\caption{Local transformations applied to a graph.}
\label{fig:moves}
\end{figure}

It is common to consider the space of edge-weightings of $G$ modulo gauge equivalence, and it is easy to see that local transformations are well-defined on this level.  Both types of local moves can be performed in either direction, where going from right to left requires first applying gauges to make the indicated edge weights equal to $1$.  The second local transformation when applied from right to left is called \emph{degree two vertex removal}.

The last bit of background we need are the basics of Kasteleyn theory, see \cite{Ken09} for a more detailed exposition.  For a planar bipartite graph $G=(B \cup W,E)$, call a map $\epsilon : E \to \{\pm 1\}$ a set of \emph{Kasteleyn signs} if
\begin{itemize}
\item each $4k$-gon face of $G$ has an odd number of $-1$'s on its boundary, while 
\item each $(4k+2)$-gon face of $G$ has an even number of $-1$'s on its boundary.
\end{itemize}
If $G$ is finite then a set of such signs always exists, and any two choices of Kasteleyn signs differ by a gauge transformation.  If a general edge-weighting of $G$ is given, the associated \emph{Kasteleyn matrix} $K$ is defined as follows.  It has rows and columns indexed by $B$ and $W$ respectively. If $b\in B$ and $w\in W$, then $K_{bw}$ equals the sum over all edges between them of the weights of these edges multiplied by the Kasteleyn signs of the edges. In particular $K_{bw} = 0$ if there is no edge between $b$ and $w$. The Kasteleyn matrix $K$ of a planar bipartite $G$ plays an important role in the study of the dimer model on $G$: the partition function is given by $|\det K|$ and the correlations are computed using minors of $K^{-1}$ \cite{Ken09}.

We now introduce a geometric model associated to every bipartite planar graph.   

\begin{definition} \label{def:vr}
Let $G$ be a planar bipartite graph with vertex set $B \cup W$.  For $b\in B$ let $N(b) \subseteq W$ denote its set of neighbors.  Fix a vector space $V$. A \emph{vector-relation configuration} on $G$ consists of choices of 
\begin{itemize}
\item a nonzero vector $v_w \in V$ for each $w \in W$ and
\item a non-trivial linear relation $R_b$ among the vectors $\{v_w : w \in N(b)\}$ for each $b \in B$.
\end{itemize}
In particular, each set $\{v_w : w \in N(b)\}$ must be linearly dependent.
\end{definition}

By a linear relation we mean a formal linear combination of vectors that evaluates to zero on $\{v_w : w \in N(b)\}$.  For technical reasons it is best to allow $G$ to have multiple edges in which case the $N(b)$ are understood to be multisets and a given vector can appear multiple times in a given relation.  We often ignore this possibility, either implicitly or by assuming $G$ to be reduced (a certain condition that implies it lacks multiple edges). A useful way to deal with multiple edges is to use the classical reduction rule of collapsing parallel edges and adding their weights.

\begin{definition} \label{def:gauge}
Consider a vector-relation configuration on a graph $G$ as above and suppose $\lambda \neq 0$.  The \emph{gauge transformation} by $\lambda$ at a black vertex $b \in B$ scales the relation $R_b$ by $\lambda$ (and keeps all other vectors and relations the same).  The \emph{gauge transformation} by $\lambda$ at a white vertex $w \in W$ scales $v_w$ by $1/\lambda$ and scales the coefficient of $v_w$ by $\lambda$ in each relation in which it appears to compensate.  Two vector-relation configurations are called \emph{gauge equivalent} if they are related by a sequence of gauge transformations.
\end{definition}

We now wish to define dynamics with the same combinatorics as local transformations for weighted bipartite graphs, but operating on our vector and relation data rather than on edge weights.  If $R$ is a relation among vectors $\{u_1,\ldots, u_k, v_1,\ldots, v_l\}$ let $R|_{u_1\cdots u_k}$ denote the linear combination of $u_1,\ldots, u_k$ appearing in $R$.  This combination may be formal or not depending on context.  For instance, as formal linear combinations we have $R|_{u_1\cdots u_k} + R|_{v_1\cdots v_l} = R$ while as vectors we have $R|_{u_1\cdots u_k} + R|_{v_1\cdots v_l} = 0$ since $R$ evaluates to $0$.

First consider urban renewal, as pictured in Figure \ref{fig:urban}.  We need to define the vectors and relations at the new vertices.  Let $u_1 = R_1|_{v_1v_2}$ and $u_2 = R_2|_{v_1v_2}$.  Note that $u_1$ and $u_2$ are both given as linear combinations of $v_1,v_2$, so if the coefficient matrix is nonsingular we can formally solve for each $v_i$ in terms of $u_1$ and $u_2$.  Moving the $u_j$ terms to the other side we get a linear relation $S_i$ among $v_i$, $u_1$, and $u_2$ for $i=1,2$.  In short, the $u_i$, $v_i$ and $S_i$ are consistent with being part of a vector-relation configuration on the new graph.  As a final step $R_1$ is modified to reflect that a linear combination of $v_1,v_2$ has been replaced by $1u_1$ and similarly with $R_2$.  Explicitly, these new relations are
\begin{displaymath}
R_i' = (R_i - R_i|_{v_1,v_2}) + 1u_i.
\end{displaymath}
Note that if the matrix mentioned above is singular then urban renewal is not defined on the configuration.

\begin{figure}
\centering
\includegraphics[height=1.5in]{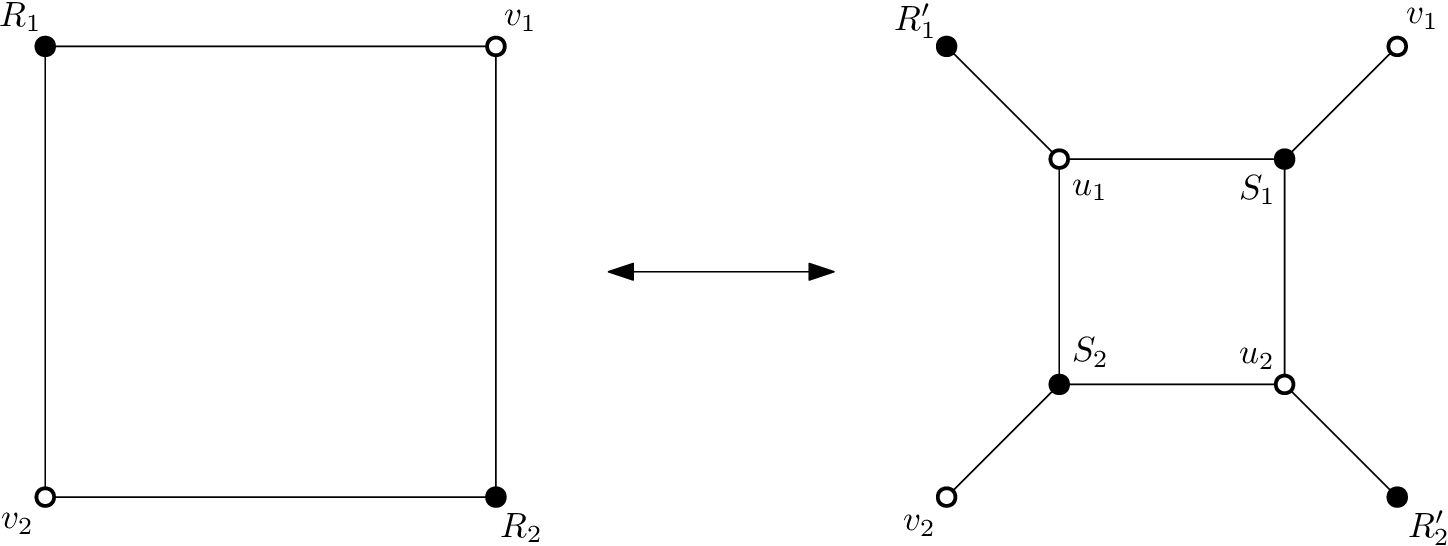}
\caption{The vector-relation version of urban renewal.}
\label{fig:urban}
\end{figure}

Next consider degree two vertex addition, as pictured in Figure \ref{fig:degree2}.  First suppose we are adding a degree two black vertex.  It is natural to set the new vectors equal to each other and to the old vector, i.e. $v = w = u$.  We then get a relation $T = 1v - 1w$.  The nearby relations do not need to be modified at all.  On the other hand, suppose we are adding a degree two white vertex.  Choose as the new vector $w = R|_{u_1\cdots u_k} = -R|_{v_1 \cdots v_l}$.  We get the relation $S$ by starting with $R$ and replacing $R|_{v_1\cdots v_l}$ with $-1w$.  Similarly $T$ is obtained from $R$ by replacing $R|_{u_1\cdots u_k}$ with $1w$.

\begin{figure}
\centering
\includegraphics[height=2in]{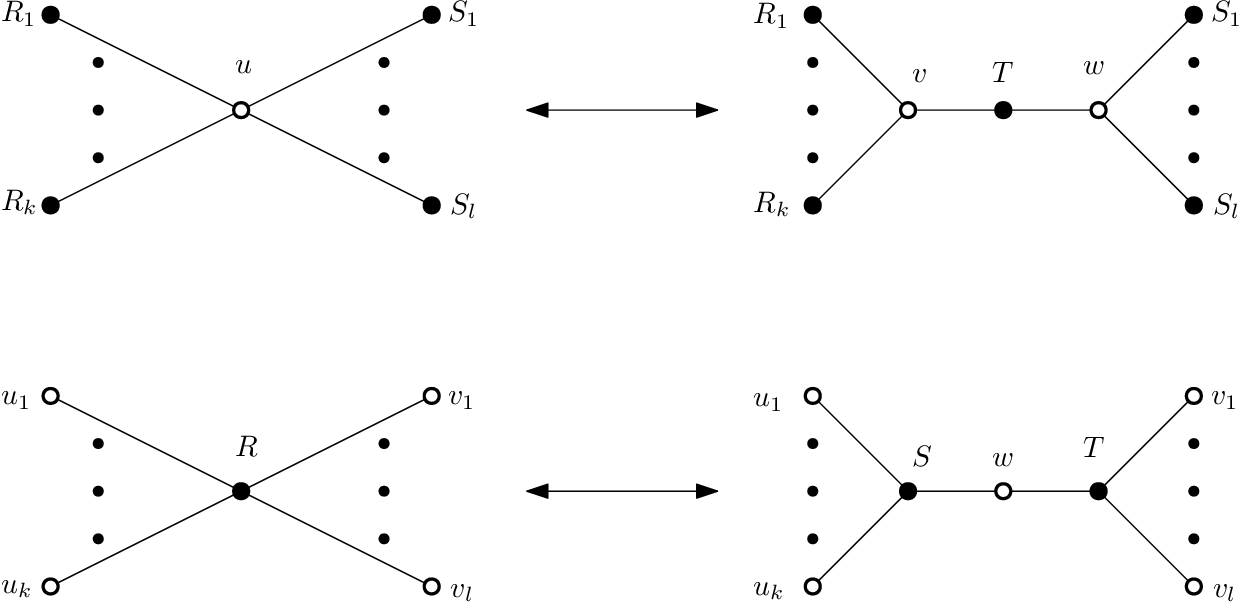}
\caption{The vector-relation version of degree two vertex addition.}
\label{fig:degree2}
\end{figure}

As with classical local transformations, these operations preserve gauge equivalence and can be run in both directions.  Thus, gauge equivalence classes of vector-relation configurations will serve as our main object of study.

\section{Vector-relation configurations} \label{sec:config}
In this section we develop the theory of vector-relation configurations on general planar bipartite graphs as in Definitions \ref{def:vr} and \ref{def:gauge}.  To that end, let $G = (B \cup W, E)$ be a planar bipartite graph.  We will denote a vector-relation configuration on $G$ by $(\v,\Rvec)$ (or sometimes just $\v$ for short) where $\v = (v_w)_{w \in W}$ and $\Rvec = (R_b)_{b \in B}$.

\subsection{Constructing the edge weights} \label{subsec:edge}
For $b \in B$ and $w \in W$, let $K_{bw}$ denote the coefficient of $v_w$ in $R_b$, understood to be $0$ if $b,w$ are not adjacent in $G$.  Performing local moves sometimes requires $K_{bw} \neq 0$ for certain $\overline{bw} \in E$, so we add that assumption when needed.  If $G$ is finite then we can view $K$ as a $|B|$-by-$|W|$ matrix.  Gauge transformations correspond to multiplying rows and/or columns of $K$ by nonzero scalars.  

The matrix $K$ plays the part of the Kasteleyn matrix (see Section \ref{sec:transform}) in the dimer model. Here the signs are already built into the entries of the matrix, and we need to remove them to obtain the weights. Fix a choice of Kasteleyn signs $\epsilon_{bw} = \pm 1$  for $\overline{bw} \in E$.  Let $\wt(e) = \epsilon_{bw}K_{bw}$ for each edge $e = \overline{bw}$ of $G$.  The $\wt(e)$ play the part of the edge weights in the classical story of local transformations of the dimer model.  As previously mentioned, in the planar case any two choices of Kasteleyn signs are gauge equivalent so the gauge class of the result depends only on the gauge class of $(\v, \Rvec)$.

\begin{remark}
The data of a gauge class of non-zero edge weights is equivalent to what Goncharov and Kenyon refer to as a trivialized line bundle with connection on $G$ \cite{GK}.
\end{remark}

\begin{proposition} \label{prop:moves}
Let $(\v, \Rvec)$ be a vector-relation configuration on $G$.  Apply a local transformation to obtain a new configuration $(\v', \Rvec')$ on $G'$.  Then the weight functions associated to these two configurations are related by a classical local transformation of the dimer model.
\end{proposition}

\begin{proof}
First suppose the operation is urban renewal, and adopt the notation of Figure \ref{fig:urban}.  Suppose the initial relations are
\begin{align*}
R_1 &= \tilde{a}v_1 + \tilde{d}v_2 + \ldots \\
R_2 &= \tilde{b}v_1 + \tilde{c}v_2 + \ldots
\end{align*}
so by definition $u_1 = \tilde{a}v_1 + \tilde{d}v_2$ and $u_2 = \tilde{b}v_1 + \tilde{c}v_2$.  We assume when doing urban renewal that the $v_1,v_2$ can be recovered from $u_1,u_2$, i.e. that $\tilde{a}\tilde{c} - \tilde{b}\tilde{d} \neq 0$.  In this case,
\begin{align*}
v_1 &= \frac{\tilde{c}u_1 -\tilde{d}u_2}{\tilde{a}\tilde{c} - \tilde{b}\tilde{d}} \\
v_2 &= \frac{-\tilde{b}u_1 + \tilde{a}u_2}{\tilde{a}\tilde{c} - \tilde{b}\tilde{d}} 
\end{align*}
The new relations are
\begin{align*}
S_1 &= v_1 + \tilde{c}'u_1 + \tilde{d}'u_2 \\
S_2 &= v_2 + \tilde{b}'u_1 + \tilde{a}'u_2 
\end{align*}
where
\begin{equation} \label{eq:urbanSigned}
\tilde{a}' = \frac{-\tilde{a}}{\tilde{a}\tilde{c} - \tilde{b}\tilde{d}}, \quad 
\tilde{b}' = \frac{\tilde{b}}{\tilde{a}\tilde{c} - \tilde{b}\tilde{d}}, \quad
\tilde{c}' = \frac{-\tilde{c}}{\tilde{a}\tilde{c} - \tilde{b}\tilde{d}}, \quad
\tilde{d}' = \frac{\tilde{d}}{\tilde{a}\tilde{c} - \tilde{b}\tilde{d}}.
\end{equation}

Let $a,b,c,d, a',b',c',d'$ be the edge weights obtained by multiplying the associated coefficients by Kasteleyn signs.  The notation has been chosen so that these weights correspond to edges in the manner indicated in Figure \ref{fig:moves}.  On the left is a quadrilateral face which should have an odd number of $-1$'s.  Applying gauge we can assume specifically $a = -\tilde{a}$, $b=-\tilde{b}$, $c=-\tilde{c}$, and $d=\tilde{d}$.  It is consistent on the right to have the edge labeled $b'$ be negative, all other pictured edges positive, and all edges outside the picture keeping their original signs.  So we put $a' = \tilde{a}'$, $b'=-\tilde{b}'$, $c'=\tilde{c}'$, and $d'=\tilde{d}'$.  Applying this substitution to \eqref{eq:urbanSigned} verifies that the edge weights evolve according to \eqref{eq:urban}, as desired.

Now suppose the transformation is degree $2$ vertex addition.  There is a natural injection from edges of $G$ to edges of $G'$, and the definitions are such that coefficients living on these edges are all unchanged.  Fixing Kasteleyn signs on $G$, we can get valid signs on $G'$ by keeping the signs of all old edges and giving the two new edges opposite signs from each other.  If the new vertex is black (see top of Figure \ref{fig:degree2}), the opposite signs are reflected in the new relation $T = 1v-1w$.  If instead it is white (bottom of Figure \ref{fig:degree2}) we have that the new vector $w$ appears with coefficient $-1$ in $S$ and $+1$ in $T$, so again the signs are opposite.  In both cases, the unsigned weights of both new edges equal $1$ in agreement with the bottom of Figure \ref{fig:moves}.
\end{proof}

Note that the map from vector-relation configurations to edge weightings on $G$ has only been defined in the one direction.  Before moving on to applications, we briefly discuss the reverse problem.  Suppose a planar bipartite graph $G = (B \cup W, E)$ is given with edge weights.  Applying Kasteleyn signs we obtain formal relations.  One approach to getting the vectors is to start with $|W|$ independent vectors and quotient the ambient space by these relations.  The resulting configuration is the most general with these edge weights in the sense that any other will be a projection of it.  In particular, assuming highest possible dimension the configuration is unique up to linear isomorphism.  We explore this construction in the plabic graph case in Section \ref{sec:plabic}.

A more difficult matter is the existence of a configuration for given edge weights.  A fundamental family of examples comes from taking $G$ to be balanced (same number of white and black vertices) on a torus.  In this case, the construction from the previous paragraph applied to generic edge weights would produce a trivial configuration with all vectors equals to $0$.  A partial remedy would be to allow twisted configurations in the spirit of twisted polygons in the theory of the pentagram map, which is the approach developed in \cite{AGeR}. 

\subsection{The face weights} \label{sec:face}
For a non-zero edge weighting on $G$, the basic gauge invariant functions are the \emph{monodromies} around closed cycles.  The monodromy of a cycle is the product of edge weights along the cycle taken alternately to the power $1$ and $-1$.  We can pull these quantities back to get gauge invariant functions of vector-relation configurations.  

We focus on the case of the monodromy around a single face $F$ of $G$.  Suppose $F$ is a $2m$-gon and that the vertices on its boundary in clockwise order are $w_1,b_1,\ldots, w_m,b_m$.  The \emph{face weight} of the face $F$ of a vector-relation configuration is
\begin{equation} \label{eq:face}
Y_F = (-1)^{m-1}\frac{K_{b_1w_1}K_{b_2w_2}\cdots K_{b_mw_m}}{K_{b_1w_2}K_{b_2w_3}\cdots K_{b_mw_1}}.
\end{equation}
The sign accounts for the product of Kasteleyn signs around the face.  In other words, we have arranged it so that this face weight equals the one defined in terms of edge weights in the corresponding weighted graph.

\begin{proposition} \label{prop:face}
Under an urban renewal move, the face weights of a vector-relation configuration evolve as in Figure \ref{fig:face}.  The face weights are unchanged by degree $2$ vertex addition/removal.
\end{proposition}

\begin{proof}
The formulas follow from the case of classical local transformations, for which they are standard, see e.g. \cite[Theorem 4.7]{GK}.
\end{proof}

\begin{figure}
\centering
\includegraphics[height=1.5in]{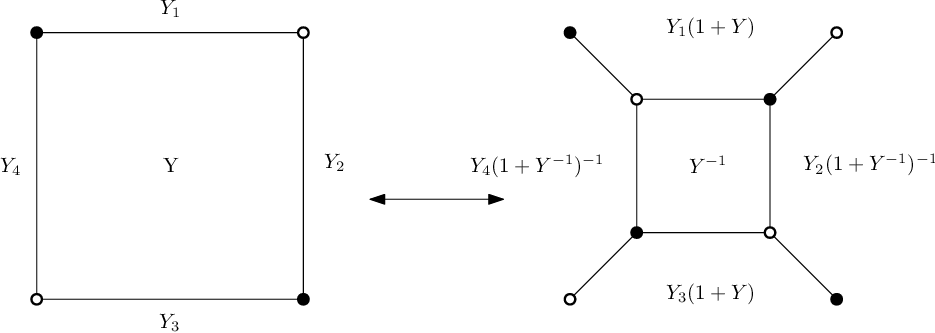}
\caption{The evolution equation for face weights}
\label{fig:face}
\end{figure}

To simplify some formulas, we introduce a vector associated to each black vertex $b$ of a given face $F$.  Suppose $w$ and $w'$ are the neighbors of $b$ around $F$.  Given a vector relation configuration we define $v(F,b) = R_b|_{ww'} = K_{bw}v_w + K_{bw'}v_{w'}$.  The $v_w$ and $v(F,b)$ for $w$ and $b$ around $F$ contain the data needed to calculate $Y_F$.  Moreover, if $F$ is a quadrilateral and $b,b'$ its black vertices, then the two new vectors arising from urban renewal at $F$ are $v(F,b)$ and $v(F,b')$.  

\subsection{Projective dynamics}
By placing an additional assumption on our configurations, we can obtain an elegant model for the gauge classes in terms of projective geometry.  Recall a set of vectors is called a \emph{circuit} if it is linearly dependent but each of its proper subsets is linearly independent.  Say that a vector-relation configuration is a \emph{circuit configuration} if each set $\{v_w : w \in N(b)\}$ is a circuit for $b \in B$.  For each $w \in W$, let $P_w$ equal the span of $v_w$, considered as a point in the projective space $\mathbb{P}(V)$.

\begin{proposition}
The gauge class of a circuit configuration is uniquely determined by the configurations of points $P_w \in \mathbb{P}(V)$ for $w \in W$.
\end{proposition}

\begin{proof}
Suppose two circuit configurations both give rise to the points $P_w$.  Then the vectors agree up to scale, so we can gauge to get the vectors $v_w$ to agree exactly.  It remains to show that for each $b \in B$ the relations $R_b$ and $R'_b$ on $\{v_w : w \in N(b)\}$ of the two configurations agree up to scale.  If not one could find a linear combination of $R_b$ and $R_b'$ with a zero coefficient and a nonzero coefficient, violating the circuit condition around $b$.
\end{proof}

As usual, we take an affine chart to visualize $\mathbb{P}(V)$ as an affine space of dimension one less than $V$.  From this point of view, a circuit of size $d$ consists of $d$ points contained in a $d-2$ dimensional space with each proper subset in general position (e.g. $4$ points on a plane of which no $3$ are collinear).

We next describe how local transformations look on the level of the points $P_w$.  For $F$ a face of $G$ and $w_1,b,w_2$ three consecutive vertices on the boundary of $F$, define
\begin{equation} \label{eq:PFb}
P(F,b) = \langle P_{w_1}, P_{w_2} \rangle \cap \langle \{P_w : w \in N(b) \setminus \{w_1,w_2\}\} \rangle
\end{equation}
where $\langle \cdot \rangle$ denotes the affine span of a set of points.  If $b$ has degree $d$ then by the preceding discussion the right hand side is a transverse intersection inside a $d-2$ space of a line and a $d-3$ space.  So $P(F,w)$ is indeed a point.

\begin{proposition} \label{prop:evolution}
Suppose we have a circuit configuration on $G$ consisting of points $P_w$ and that $F$ is a quadrilateral face with vertices $w_1,w_2$ and $b_1,b_2$ each having degree at least $3$.  If $P(F,b_1) \neq P(F,b_2)$, then
\begin{itemize}
\item urban renewal of the configuration is defined at $F$,
\item the result of urban renewal is a circuit configuration, and
\item urban renewal at $F$ constructs the point $P(F,b_i)$ at the new white vertex closer to $b_i$ for $i=1,2$.
\end{itemize}
\end{proposition}

\begin{proof}
First, we show that $P(F,b_1)$ is the class of $v(F,b_1)$ in projective space (and similarly for $P(F,b_2)$ and $v(F,b_2)$).  Indeed, $v(F,b_1)$ is by definition the linear combination of $v_{w_1}$ and $v_{w_2}$ appearing in $R_{b_1}$.  Applying $R_{b_1}$, one can equivalently express $v(F,b_1)$ as a linear combination of $\{v_w : w \in N(b_1) \setminus \{w_1,w_2\}\}$.  We get that $v(F,b_1)$ is on the intersection of two subspaces in a way that exactly projectivizes to the formula \eqref{eq:PFb} (note the circuit condition implies that $v(F,b_1) \neq 0$).

Now, since the $v(F,b_i)$ projectivize to distinct points, they must be linearly independent.  These vectors play the role of $u_1,u_2$ in Figure \ref{fig:urban} and their being independent is equivalent to the non-degeneracy condition needed to perform urban renewal.  It also follows that the $P(F,b_i)$ are the projectivizations of the new vectors produced by urban renewal.  All that remains is to prove the second assertion.

The circuit condition in the original graph implies that all coefficients of all relations at black vertices are nonzero.  Moreover, we know $v_{w_1}$ and $v_{w_2}$ are independent, e.g. by the circuit condition at $b_1$ together with the fact that $b_1$ has degree at least $3$.  Now, urban renewal produces two new black vertices (the ones labeled $S_i$ in Figure \ref{fig:urban}), and modifies the neighborhood of two others (the ones labeled $R_i'$).  First consider a new black vertex, say the one adjacent to the vectors $v(F,b_1)$, $v(F,b_2)$, and $v_{w_1}$.  We have already argued that the first two are independent.  Recall that $v(F,b_1) = R_{b_1}|_{v_{w_1}v_{w_2}}$ and by the facts at the beginning of this paragraph is independent of $v_{w_1}$.  Similarly $v(F,b_2)$ and $v_{w_1}$ are linearly independent.  So the circuit condition holds at this vertex.

Lastly consider one of the black vertices with a modified neighborhood, say the one originally called $b_1$.  The set of vectors at neighboring vertices is the same after urban renewal as before except that $v_{w_1}$ and $v_{w_2}$ have been removed, and $v(F,b_1)$ has been added.  Were there a linear dependence among a proper subset of these vectors introduced, it would have to include $v(F,b_1)$.  However, $v(F,b_1)$ is a linear combination of $v_{w_1}$ and $v_{w_2}$, so this would imply a dependence in the original graph contradicting the circuit condition there.
\end{proof}

\begin{proposition}
Suppose we have a circuit configuration on $G$ consisting of points $P_w$.  Consider a degree $2$ vertex addition move from $G$ to $G'$.  If the added degree $2$ vertex $b$ is black then the point $P$ at the white vertex of $G$ that got split is placed at both neighbors of $b$ in $G'$.  If the added degree $2$ vertex $w$ is white, let $P_1,\ldots, P_k$ and $Q_1,\ldots, Q_l$ be the points at the neighbors of the black vertex of $G$ that got split, following the template of the bottom of Figure \ref{fig:degree2}.  Then the new point that gets placed at $w$ is
\begin{displaymath}
\langle P_1,\ldots, P_k \rangle \cap \langle Q_1,\ldots, Q_l \rangle.
\end{displaymath}
In both cases, we still have a circuit configuration on $G'$.
\end{proposition}

\begin{proof}
The black degree $2$ vertex addition case follows directly from the definitions.  The proof in the white case follows the same approach as the proof of Proposition \ref{prop:evolution}.
\end{proof}

The circuit condition is preserved by the removal of a degree $2$ black vertex, since the neighborhood of each remaining black vertex did not get changed. Note however that in general, the circuit condition is not preserved by the removal of a degree $2$ white vertex. This is the case for example if the two black vertices adjacent to the degree $2$ white vertex have degrees $d_1\leq \dim V+1$ and $d_2\leq \dim V+1$, with $d_1+d_2\geq \dim V +4$. Nevertheless, in all the examples we will consider in Section~\ref{sec:examples}, the circuit condition will be preserved even by removals of degree $2$ white vertices, provided we start with a generic configuration, see Remark~\ref{rem:generic}.

To sum up, for each planar bipartite graph $G$ we have a projective geometric dynamical system dictated by the corresponding dimer model.  The state of the system is given by a choice of a point in projective space at each white vertex so that the points neighboring each black vertex form a circuit.  The points (and the graph) evolve under local transformations, the most interesting of which is urban renewal as described by Proposition \ref{prop:evolution} and formula \eqref{eq:PFb}.

We will see that many systems, some in the pentagram map family some not, fit in this framework.  For each such system we get for free the set of face weights $Y_F$ and their corresponding evolution equations as in Proposition \ref{prop:face}.  These variables are easy to define in a projectively natural way.  Suppose points $P_{1},\ldots, P_{2k}$ in an affine chart are given with the triples $\{P_1,P_2,P_3\}$, $\{P_3,P_4,P_5\}$, \ldots, $\{P_{2k-1},P_{2k},P_1\}$ all collinear.  The \emph{multi-ratio} (called a \emph{cross ratio} for $k=2$ and a \emph{triple ratio} for $k=3$) of the points is
\begin{displaymath}
[P_1,\ldots, P_{2k}] = \frac{P_1-P_2}{P_2-P_3} \frac{P_3-P_4}{P_4-P_5} \cdots \frac{P_{2k-1}-P_{2k}}{P_{2k}-P_1}.
\end{displaymath}
Each individual fraction involves $3$ points on a line and is interpreted as a ratio of signed distances. It is well-known that this ratio is independent of the chart and invariant under projective transformations, see e.g. \cite[Theorem 9.11]{BS}.

\begin{proposition} \label{prop:faceweights}
Suppose we have a circuit configuration on $G$ consisting of points $P_w$.  Let $F$ be a face with boundary cycle $w_1, b_1, w_2, b_2, \ldots, w_m, b_m$ in clockwise order.  In terms of the points $P_w$, the face weight of $F$ equals
\begin{displaymath}
Y_F = (-1)^{m-1}[P_{w_1}, P(F,b_1), P_{w_2}, P(F,b_2), \ldots, P_{w_m}, P(F,b_m)]^{-1}.
\end{displaymath}

\end{proposition}

\begin{proof}
By definition we have lifts $v_{w_i}$ of $P_{w_i}$ and $v(F,b_i)$ of $P(F,b_i)$ such that
\begin{displaymath}
v(F,b_i) = K_{b_iw_i}v_{w_i} + K_{b_iw_{i+1}}v_{w_{i+1}}.
\end{displaymath}
Applying gauge we can assume all $2m$ of these vectors lie in some affine hyperplane.  Then applying a linear functional with constant value $1$ on said hyperplane to the previous yields
\begin{displaymath}
1 = K_{b_iw_i} + K_{b_iw_{i+1}}.
\end{displaymath}
As such the above can be rewritten
\begin{displaymath}
K_{b_iw_{i+1}}(v(F,b_i) - v_{w_{i+1}}) = K_{b_iw_i}(v_{w_i}-v(F,b_i)).
\end{displaymath}
Viewing the $P$'s as points on the hyperplane this shows
\begin{displaymath}
\frac{P(F,b_i) - P_{w_{i+1}}}{P_{w_i}-P(F,b_i)} = \frac{K_{b_iw_i}}{K_{b_iw_{i+1}}}.
\end{displaymath}
Multiplying across all $i$ produces the reciprocal of the multi-ratio on the left and the defining expression \eqref{eq:face} for the face weights on the right.
\end{proof}

\section{Examples} \label{sec:examples}
In this section, we consider several projective geometric systems from the literature, and explain how they fit in our framework.  For each we identify the appropriate bipartite graph as well as the sequence of local transformations realizing the system.  In some cases we also explicitly work out the associated dynamics on the face weights.

\begin{remark}\label{rem:generic}
In order to work on the level of projective geometry, all configurations in this Section are assumed to be circuit configurations.  Moreover, each individual system is only defined for a subset of such configurations. Indeed, every urban renewal move requires a certain non-degeneracy condition, see Proposition \ref{prop:evolution}. Furthermore, for all these examples, the removal of degree $2$ white vertices will preserve the circuit condition only if one requires a genericity assumption on the starting configuration. One nice application of defining a system this way is one can obtain a large family of inputs for which all iterates are guaranteed to be defined, namely those with generic positive edge weights.
\end{remark}

\begin{remark}
The examples in this Section all take place on infinite bipartite graphs in the plane.  In some we assume the points of the configuration are biperiodic with respect to some lattice in the plane.  One can just as well impose as boundary conditions that the face weights be biperiodic, but not the points themselves.  This choice lines up with the dimer model on the torus, and one in principle can use \cite{GK} to prove a lot about such systems (Liouville integrability, spectral curve, combinatorial formulas for conserved quantities, ...). Such an approach has been implemented in \cite{AGeR} for some dynamics on spaces of polygons phrased in terms of vector-relation configurations. Another approach to integrability was proposed by Gekhtman, Shapiro, Tabachnikov, and Vainshtein \cite{GSTV16} for pentagram maps ; the connection with the approach of \cite{GK} was recently explained in \cite{Izosimov}.
For other examples below, the biperiodic face weights condition gives special cases that to our knowledge have not been rigorously studied.
\end{remark}

\subsection{The pentagram family}
\begin{example} \label{ex:laplace}
The Laplace-Darboux system \cite{D97} operates on a $2$-dimensional array of points in $\mathbb{P}^3$ for which the points of each primitive square are coplanar.  It is convenient to index the points as $P_{i,j}$ for $i,j \in \mathbb{Z}$ with $i+j$ even.  The centers of the squares are then $(i,j)$ with $i+j$ odd so the condition is
\begin{equation} \label{eq:laplaceCond}
P_{i,j-1}, P_{i-1,j}, P_{i+1,j}, P_{i,j+1} \textrm{ coplanar for $i+j$ odd}.
\end{equation}
The system produces a new array of points $Q_{i,j}$ for $i+j$ odd defined by
\begin{displaymath}
Q_{i,j} = \langle P_{i,j-1}, P_{i+1,j} \rangle \cap \langle P_{i-1,j}, P_{i,j+1}  \rangle.
\end{displaymath}

To state Laplace-Darboux dynamics in our language take the infinite square grid graph $G = (\mathbb{Z}^2, E)$, which is bipartite with white vertices being those $(i,j)$ with $i+j$ even.  Place the points $P_{i,j}$ above at the white vertices.  For each black vertex $(i,j)$ with $i+j$ odd, the circuit condition says that the $4$ neighboring points should be coplanar, which is precisely \eqref{eq:laplaceCond}.

To evolve the system, perform urban renewal at each face whose upper left corner is black.  Figure \ref{fig:laplace} shows a local picture.  Taking $F,b$ as in the picture, one of the new points is
\begin{displaymath}
P(F,b) = \langle P_{2,0}, P_{3,1} \rangle \cap \langle P_{1,1}, P_{2,2} \rangle = Q_{2,1}.
\end{displaymath}
Eliminating all degree $2$ vertices in the resulting picture recovers the square lattice except with the colors of vertices reversed.  The surviving points are precisely the $Q_{i,j}$.
\end{example}

\begin{figure}
\centering
\includegraphics[height=2in]{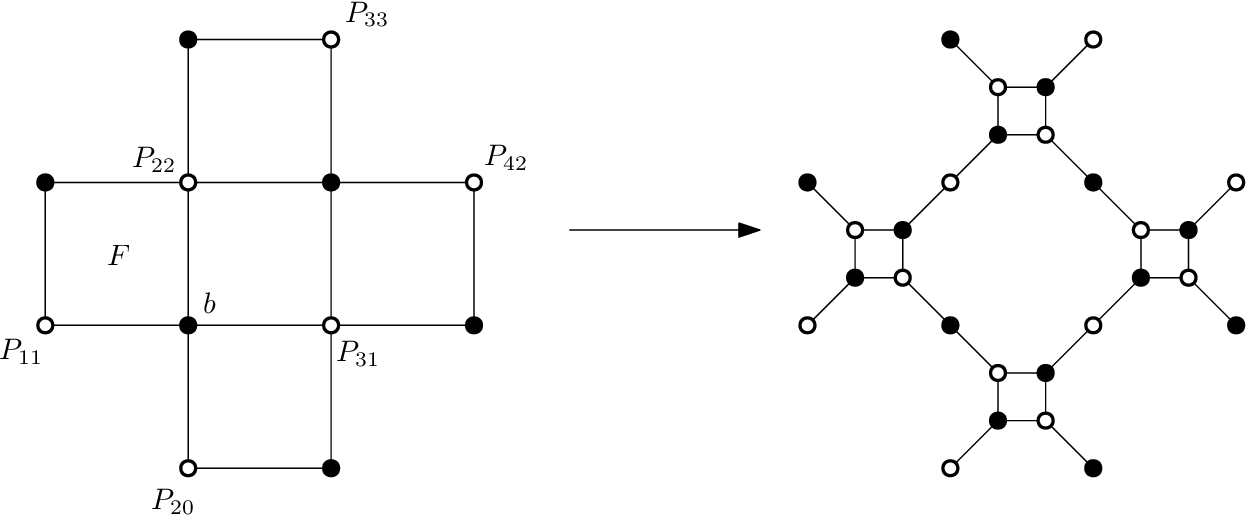}
\caption{The local transformations which, when followed by degree $2$ vertex removals, realize Laplace-Darboux dynamics.}
\label{fig:laplace}
\end{figure}

\begin{example}
The pentagram map takes as input a polygon in $\mathbb{P}^2$ with vertices $A_i$ for $i \in \mathbb{Z}$ and outputs the polygon with vertices 
\begin{displaymath}
B_i = \langle A_{i-1}, A_{i+1} \rangle \cap \langle A_{i}, A_{i+2} \rangle.
\end{displaymath}
This operation can be seen as a reduction of Laplace-Darboux dynamics.  Indeed, one can check that letting
\begin{align*}
P_{i,j} &= A_{(i+3j)/2}, \textrm{ $i+j$ even} \\
Q_{i,j} &= B_{(i+3j-1)/2}, \textrm{ $i+j$ odd}
\end{align*}
gives an input-output pair for Laplace-Darboux.  Note that $P_{i,j} = P_{i-3,j+1}$ and moreover if $A$ is a closed $n$-gon meaning $A_{i+n} = A_i$ then $P_{i,j} = P_{i+2n,j}$.

As the bipartite graph for Laplace-Darboux was the square grid on $\mathbb{Z}^2$, the correct choice for the pentagram map is the quotient of this graph by the lattice generated by $(-3,1)$ and $(2n,0)$.  This is a bipartite graph on a torus.  Point $A_i$ labels (the class of) the vertex $(2i,0)$.  The relations are of the form ``$A_{i-1}, A_i, A_{i+1}, A_{i+2}$ coplanar'' which explains why the whole configuration must be in a plane.  Finally, the local transformations take the same form as for Laplace-Darboux.

The face weights of a polygon are precisely the $y$-parameters as defined in \cite{G11}.  As an example, Figure \ref{fig:penta} gives a portion of the bipartite graph.  Applying Proposition \ref{prop:faceweights}, the variable at the face labeled $F$ is
\begin{displaymath}
Y_F = -[A_3, \langle A_3, A_4 \rangle \cap \langle A_1, A_2 \rangle, A_4, \langle A_3, A_4 \rangle \cap \langle A_5, A_6 \rangle]^{-1}.
\end{displaymath}

Although this algebraic formulation of the pentagram map was known \cite{G11}, there may be other insights to be gained from the vector-relation perspective.  For instance, if nearby vertices of a polygon come together it creates a singularity for the pentagram map dynamics.  Keeping track of the coefficients of the relation satisfied by the points as they come together would be one way to try to control the behavior through the singularity.
\end{example}

\begin{figure}
\centering
\includegraphics[height=2in]{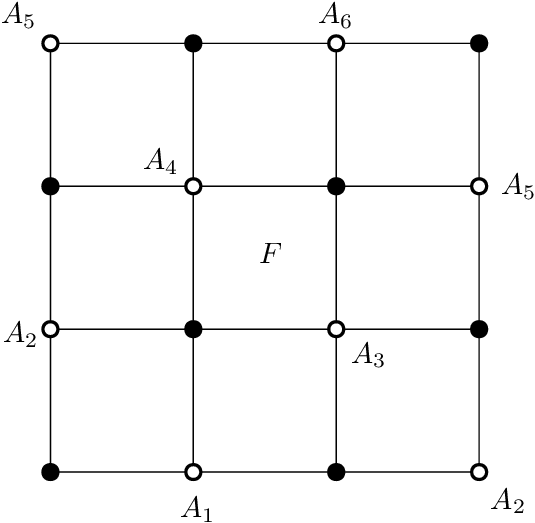}
\caption{A portion of the bipartite graph whose vector-relation dynamics coincide with the pentagram map.}
\label{fig:penta}
\end{figure}

\begin{example}
Different ways of putting the square grid graph on the torus produce different interesting systems.  The higher pentagram map of Gekhtman et al. \cite{GSTV12} is obtained by working in $\mathbb{RP}^d$ and identifying $(i,j)$ with $(i-d-1,j+d-1)$.  Indeed, the $(-i,i)$ form a set of representatives of the white vertices.  Placing a point $P_i$ at each $(-i,i)$, the neighbors of a given black vertex are labeled by $P_i, P_{i+1}, P_{i+d}, P_{i+d+1}$.  The condition that such four-tuples be coplanar is called the corrugated property and the sequence of moves from Example \ref{ex:laplace} produces points
\begin{displaymath}
\langle P_i, P_{i+d} \rangle \cap \langle P_{i+1}, P_{i+d+1} \rangle
\end{displaymath}
as in the higher pentagram map.
\end{example}

\begin{example}
The left of Figure \ref{fig:spiral} depicts one step of a certain pentagram spiral system \cite{Sch13}.  The input is a \emph{seed} consisting of five points $A_1,\ldots, A_5$ with $A_5$ lying on the line through $A_1$ and $A_4$.  The output is a new seed $A_2,\ldots, A_6$ with $A_6 = \langle A_1, A_3 \rangle \cap \langle A_2, A_5 \rangle$.  If iterated the result is a polygonal curve that spirals inwards indefinitely.

\begin{figure}
\centering
\includegraphics[height=2in]{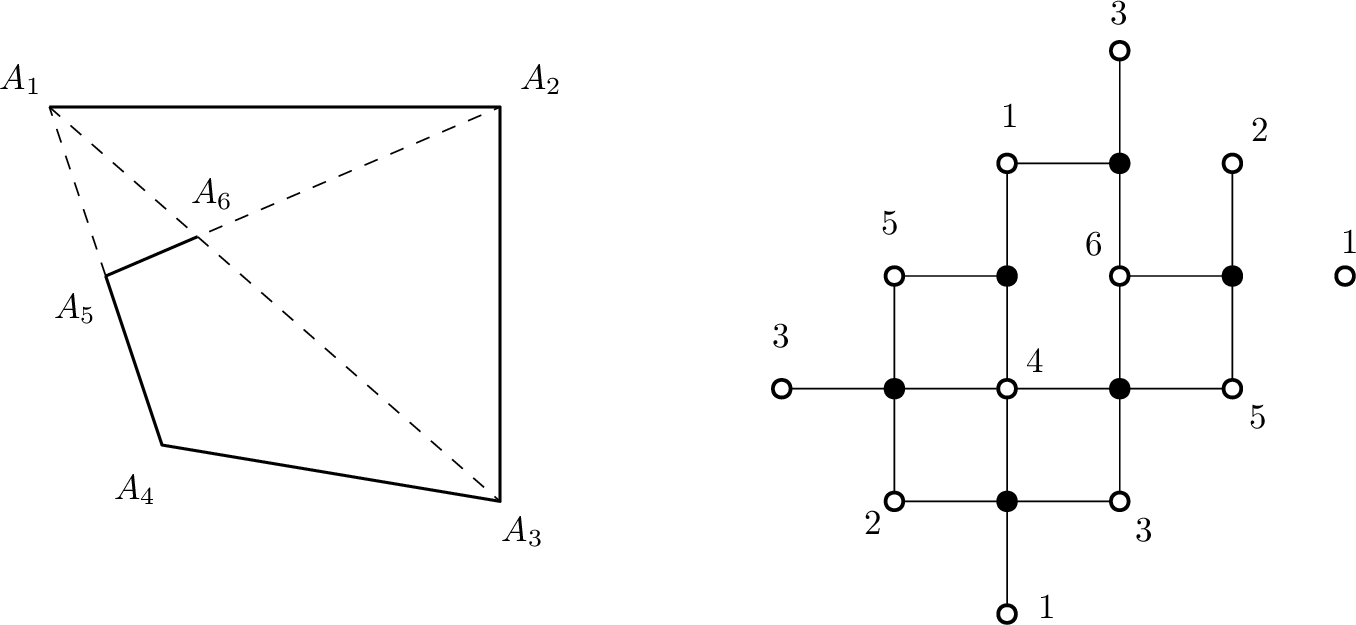}
\caption{One step of a system that produces a pentagram spiral (left) along with the associated bipartite graph (right).}
\label{fig:spiral}
\end{figure}

The right of Figure \ref{fig:spiral} shows a bipartite graph whose vector-relation dynamics captures this system.  As with the pentagram map on hexagons, the vertex set of the graph $G$ is $\mathbb{Z}^2$ modded out by the lattice generated by $(-3,1)$ and $(12,0)$.  However, $G$ does not include all of the edges from the square grid.  The figure shows exactly one copy of each edge and each black vertex, while the repeats among white vertices help to visualize how the picture repeats when lifted to $\mathbb{Z}^2$.

Place points $A_i$ for $i=1,\ldots, 6$ at the white vertices.  There are three degree $4$ black vertices which give conditions that $\{A_1,A_2,A_3,A_4\}$, $\{A_2,A_3,A_4,A_5\}$, and $\{A_3,A_4,A_5,A_6\}$ are coplanar.  As such, all six points are on a common plane.  There are also three degree $3$ black vertices implying that the triples $\{A_1,A_4,A_5\}$, $\{A_1,A_3,A_6\}$, and $\{A_2,A_5,A_6\}$ are collinear.  These match the defining conditions of the six points in the left picture.  In short, being a configuration on $G$ is equivalent to being a union of two consecutive seeds of the pentagram spiral.

We give a quick description of how to realize spiral dynamics.  There is a quadrilateral face of $G$ containing white vertices $1$ and $3$.  Urban renewal at this face followed by a degree $2$ vertex removal will produce a graph isomorphic to $G$.  The points $A_2,\ldots, A_6$ will remain and there will also be a new point $A_7 = \langle A_1, A_3 \rangle \cap \langle A_2, A_4 \rangle$, which is the next point on the spiral.  So the dynamics on the graph are equivalent to the spiral map, with the only discrepancy being that the former keeps track of six consecutive points at each time instead of five.

The graph $G$ is a special case of the dual graph to a Gale-Robinson quiver, see \cite{JMZ}.  It is likely that every sufficiently large such graph models some combinatorial type of pentagram spiral.
\end{example}

\begin{example}
The second and third authors \cite{GP16} defined a family of dynamical systems that iteratively build up certain maps from $\mathbb{Z}^2$ to a projective space termed $Y$-meshes.  Rather than give the full definition, we focus on a single illustrative example.

The \emph{rabbit map} acts on the space of triples $A,B,C$ of polygons in $\mathbb{P}^4$ satisfying the conditions
\begin{align*}
& A_{i-1}, B_{i+1}, C_i \textrm{ collinear} \\
& A_{i+1}, B_{i}, C_i \textrm{ collinear} \\
& A_{i-1}, B_{i-1}, B_{i+1}, C_{i+1} \textrm{ coplanar}
\end{align*}
for all $i \in \mathbb{Z}$.  The map takes $(A,B,C)$ to $(B,C,D)$ where
\begin{displaymath}
D_i = \langle A_{i-1}, B_{i+1} \rangle \cap \langle B_{i-1}, C_{i+1} \rangle
\end{displaymath}
for all $i$.  The vector-relation formulation of the rabbit map is given in Figure \ref{fig:rabbit}.  The black vertices correspond exactly to the conditions listed above.  Propagation is carried out by applying urban renewal for each $i$ at the square face containing both $A_{i-1}$ and $B_{i+1}$.

\begin{figure}
\centering
\includegraphics[height=2in]{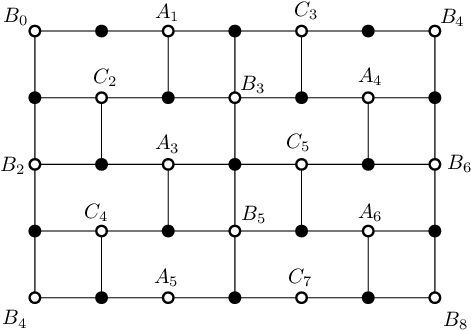}
\caption{The bipartite graph corresponding to the rabbit map.  The figure continues infinitely up and down, while the left and right sides are identified as per the labeling.}
\label{fig:rabbit}
\end{figure}

In general, a $Y$-mesh is a map $(i,j) \mapsto P_{i,j}$ from $\mathbb{Z}^2$ to a projective space such that each translate of a fixed $4$ element subset of $\mathbb{Z}^2$ maps to a quadruple of collinear points.  For instance, the rabbit map is invertible and a $Y$-mesh can be built from one of its orbits.  Begin with $P_{i,0} = A_i$, $P_{i,1} = B_i$, $P_{i,2}=C_i$ for all $i \in \mathbb{Z}$, and fill out the rest by performing the map in both directions, e.g. $P_{i,3} = D_i$.  In this example, each quadruple $P_{i-1,j}, P_{i+1,j+1}, P_{i,j+2}, P_{i,j+3}$ ends up being collinear.

We strongly suspect that every system from \cite{GP16} is a special case of vector-relation dynamics.  More precisely, we showed in the former that each such system is modeled algebraically by local transformations on a certain bipartite graph, and at least in examples the geometric dynamics can be seen to line up as well.

The vector-relation perspective represents a significant improvement in our understanding of $Y$-meshes.  As an example, in the original formulation only the cross ratio $y$-variables are easy to describe and the others require a messy case by case analysis \cite[Section 13]{GP16}.  Now we get a uniform description of all $y$-variables via Proposition \ref{prop:faceweights}.  For instance, the hexagon in Figure \ref{fig:rabbit} containing $A_3,C_2,B_3$ has weight
\begin{displaymath}
y_F = [A_3, B_2, C_2, A_1, B_3, \langle A_3, B_3 \rangle \cap \langle B_5, C_5 \rangle]^{-1}.
\end{displaymath}
A central question that is open in general is what minimum collections of points determine the $Y$-mesh and what relations they satisfy (see \cite[Section 8]{GP16} for many examples including the rabbit case).  There is hope that these questions have answers in terms of graph theoretic properties of $G$.  A result of this flavor in a different context is given in Proposition \ref{prop:matroid}.
\end{example}

\subsection{$Q$-nets} \label{sec:Qnet}

Discrete conjugate nets, or {\it {$Q$-nets}} were introduced by Doliwa-Santini \cite{DS}; we follow the exposition of Bobenko-Suris \cite{BS}.
We shall specifically be concerned with $3$-dimensional $Q$-nets, defined as follows. 

\begin{definition} \cite[Definition 2.1]{BS}
 A map $f: \mathbb Z^3 \rightarrow \mathbb R^3$ is a $3$-dimensional $Q$-net in $\mathbb R^3$ if for every $u \in \mathbb Z^3$ and for every pair of indices $i,j \in \{1,2,3\}$, points $f(u), f(u+e_i), f(u+e_j), f(u+e_i+e_j)$ are coplanar (where $e_1, e_2, e_3$ are the generators of $\mathbb Z^3$).
\end{definition}

While a $Q$-net is a static object, it is often convenient to think of it in a dynamical way as follows. For $u = (i,j,k)$ let $|u| = i+j+k$. A {\it {generation}} of vertices of a $Q$-net is the set of all $f(u)$ where $|u|=t$. Let us denote $f_t$ such $t$-th generation. Then knowing $f_t$ and $f_{t+1}$ one can construct the next generation $f_{t+2}$ as follows. Consider an elementary cube consisting of eight points $f(u+\epsilon_1 e_1+\epsilon_2 e_2 +\epsilon_3 e_3)$, where each $\epsilon_i$ is either $0$ or $1$. Assume $|u|=t-1$. Then using the six points that belong to $f_t, f_{t+1}$ one can construct three planes that have to contain $f(u+e_1+e_2+e_3) \in f_{t+2}$. Intersecting those planes we generically get the unique candidate for $f(u+e_1+e_2+e_3)$. 

The problem of parametrization of $Q$-nets, i.e. defining certain geometric quantities and giving formulas for how they evolve from generation to generation, is discussed in \cite{BS}.  The first such description goes back to the original work \cite{DS}. Our construction suggests a new way to parametrize $3$-dimensional $Q$-nets. Furthermore, since our parameters are cross-ratios of quadruples of points, it is natural to view it as parametrizing {\it {projective $Q$-nets}}, i.e. $Q$-nets considered up to projective equivalence. 

Consider three consecutive generations $f_t, f_{t+1}, f_{t+2}$ of a $Q$-net $f$. Their vertices and the edges that connect them can be conveniently visualized as a lozenge tiling dual to the Kagome lattice, see Figure \ref{fig:k1}. Vertices of each lozenge map into vertices of a face of one of the elementary cubes of a $Q$-net, and thus are coplanar. 
Thus, the geometry of the three generations of points is captured by the bipartite graph we get by placing a white vertex at each vertex of the lozenge tiling, and a black vertex at each face, see Figure \ref{fig:k1}.  The set of points of this configuration is sufficient initial data to determine the whole $Q$-net.  In fact, the rest of the $Q$-net is obtained via local transformations, which by an inductive argument boils down to the following.

\begin{figure}[h!]
    \begin{center}
    \input{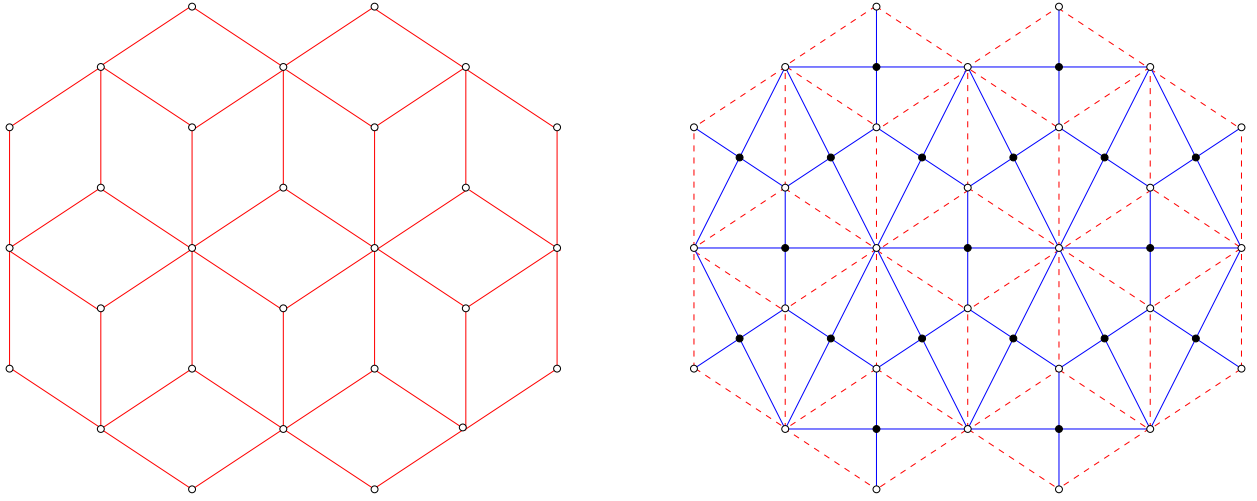}
    \end{center}
    \caption{Three generations of a $Q$-net and the associated bipartite graph}
    \label{fig:k1}
\end{figure}

\begin{proposition}
 The sequence of square moves shown in Figure \ref{fig:k2} realizes geometrically a step of time evolution of the $Q$-net transitioning from vertex $D$ to vertex $D'$ of one of the elementary cubes.
\end{proposition}

\begin{figure}[h!]
    \begin{center}
    \input{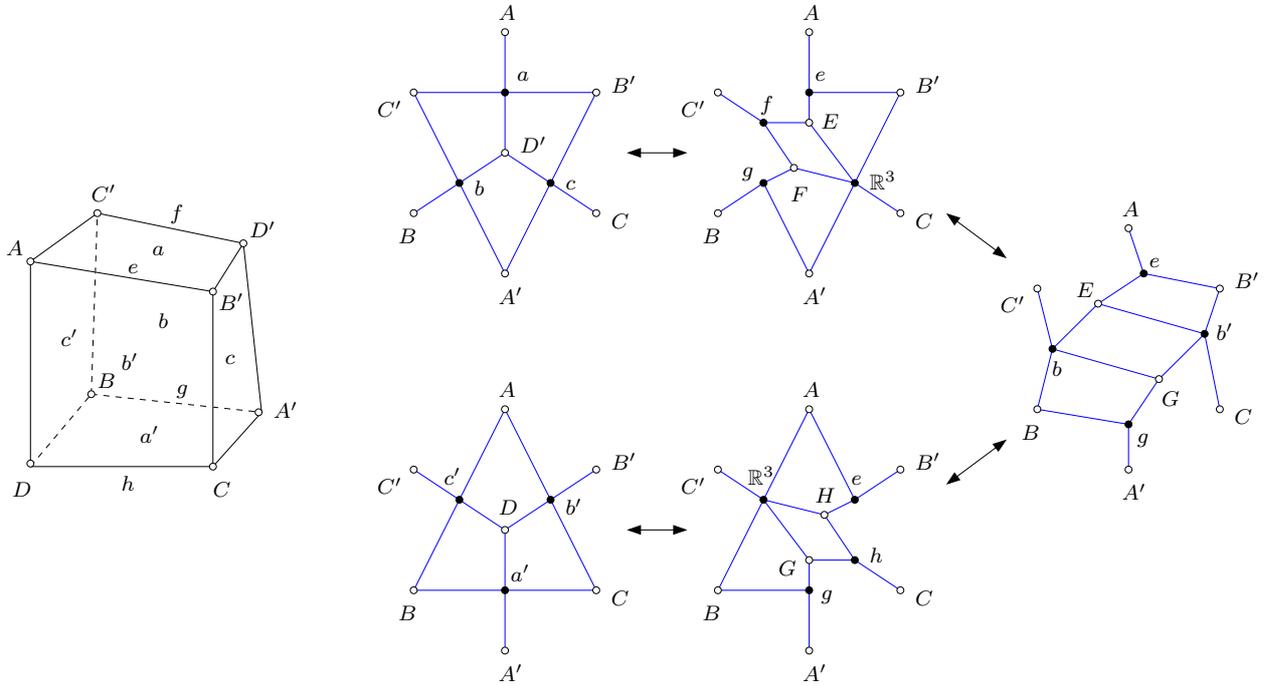}
    \end{center}
    \caption{The gentrification sequence of moves.  For convenience, each black vertex is labeled by the affine hull of the points at neighboring white vertices.}
    \label{fig:k2}
\end{figure}

\begin{proof}
 We verify the sequence of square moves using Proposition \ref{prop:evolution} on each step. For example, points $G$ and $H$ are formed by intersecting line $CD$ with affine spans of the rest of white points surrounding respectively $a'$ and $b'$, i.e. with lines $BA'$ and $AB'$. Here $E = C'D' \cap AB'$, $F = C'D' \cap BA'$, $G=A'B \cap CD$, $H=AB' \cap CD$, $a = \langle AB'D'C' \rangle$ (the affine hull of these $4$ points, which is a plane by the definition of $Q$-net), $b = \langle A'BC'D' \rangle$, $c = \langle B'CA'D' \rangle$, $e = AB'$, $f = C'D'$, $g = A'B$, $h = CD$, $a' = \langle A'BDC \rangle$, $b' = \langle AB'CD \rangle$, $c' = \langle ADBC' \rangle$.
\end{proof}

\begin{remark}
 This sequence of square moves has appeared in \cite[Figure 8]{KLRR}, without the current geometric interpretation, and under the name of {\it {star-triangle move}}. Here we introduce the name {\it {gentrification}} to emphasize its similar, but not coinciding nature with {\it {superurban renewal}} of \cite{KP}, see below. 
\end{remark}

Denote by $Q_{i,j,k}$ the vertex of the Q-net with coordinates $i,j,k$.  Let $Q_{i,j,k}^x$ be the edge connecting $Q_{i,j,k}$ with $Q_{i+1,j,k}$ (thinking of the first coordinate as the $x$-direction).  Define $Q_{i,j,k}^y$ and $Q_{i,j,k}^z$ similarly.  By the previous discussion, $3$ successive generations $f_t, f_{t+1}, f_{t+2}$ of the $Q$-net are the points of an associated circuit configuration.  Each face of the bipartite graph corresponds to an edge of the lozenge tiling, see the right of Figure \ref{fig:k1}.  Following our recipe from Proposition \ref{prop:faceweights} (we omit the details), we get formulas for the associated face weights.  They are
\begin{displaymath}
Y_{i,j,k}^x = -[Q_{i,j,k}, Q_{i,j,k}^x \cap Q_{i,j+1,k}^x, Q_{i+1,j,k}, Q_{i,j,k}^x \cap Q_{i,j,k+1}^x]^{-1}
\end{displaymath}
for $i+j+k=t$, 
\begin{displaymath}
\tilde{Y}_{i,j,k}^x = -[Q_{i,j,k}, Q_{i,j,k}^x \cap Q_{i,j,k-1}^x, Q_{i+1,j,k}, Q_{i,j,k}^x \cap Q_{i,j-1,k}^x]^{-1}
\end{displaymath}
for $i+j+k=t+1$, as well as two more copies of these formulas with the superscripts replaced by $y$ or $z$, and all subscripts cyclically shifted to the right by $1$ or $2$ spots respectively.  To sum up:

\begin{proposition}
The collection of face weights, also known as the $Y$-seed, corresponding to the above setup is
\begin{displaymath}
\{Y_{i,j,k}^* : i+j+k = t\} \cup \{\tilde{Y}_{i,j,k}^* : i+j+k = t+1\}.
\end{displaymath}
\end{proposition}

 \begin{proposition}
  The variables $Y$ evolve according to the following formulas (and their cyclic shifts):
\begin{align*}
\tilde{Y}_{i,j+1,k+1}^x &= (Y_{i,j,k}^y)^{-1} \frac{1 + Y_{i,j,k}^x + Y_{i,j,k}^y Y_{i,j,k}^x}{1 + Y_{i,j,k}^z + Y_{i,j,k}^z Y_{i,j,k}^x} \\
Y_{i+1,j,k}^x &= \tilde{Y}_{i+1,j,k}^xY_{i+1,j-1,k}^xY_{i+1,j-1,k}^y\frac{1 + Y_{i+1,j,k-1}^z + Y_{i+1,j,k-1}^zY_{i+1,j,k-1}^x}{1+Y_{i+1,j-1,k}^x + Y_{i+1,j-1,k}^xY_{i+1,j-1,k}^y}
\end{align*}
for all $i+j+k=t$.
 \end{proposition}

\begin{figure}[h!]
    \begin{center}
    \input{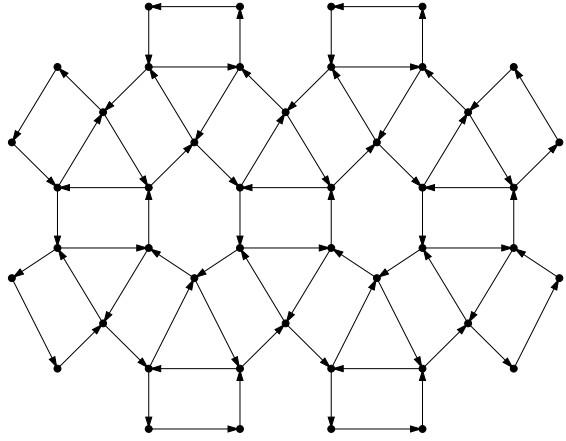}
    \end{center}
    \caption{$Q$-net quiver.}
    \label{fig:k3}
\end{figure}

\begin{proof}
 One simply follows $Y$-variable dynamics of the associated cluster algebra, whose quiver is shown in Figure \ref{fig:k3}. 
\end{proof}

\begin{remark}
 There is of course also $X$-variable cluster dynamics associated with gentrification. It is given by 
$$X_{i,j+1,k+1}^x = \frac{X_{i,j,k}^x X_{i,j+1,k}^z X_{i,j,k+1}^y + X_{i,j,k}^y X_{i,j+1,k}^z X_{i,j,k+1}^x + X_{i,j,k}^z X_{i,j+1,k}^x X_{i,j,k+1}^y}{X_{i,j,k}^y X_{i,j,k}^z}.$$
It is not clear if the $X$-variables have any geometric meaning in terms of $Q$-nets however. 
\end{remark}

\begin{remark}
Several cluster algebra descriptions of geometric systems, including $Q$-nets and discrete Darboux maps, were found independently in \cite{AGR}.  A common situation that in particular holds for $Q$-nets is that there are two distinct sets of geometric quantities that each evolve according to the (coefficient type) dynamics of the same quiver.  One of the goals of \cite{AGR} is to better understand this phenomenon.
\end{remark}

\subsection{Discrete Darboux maps}
Discrete Darboux maps were introduced by Schief \cite{Sch}; we follow the exposition of Bobenko-Suris \cite[Exercise 2.8, 2.9]{BS}.
We identify the set of edges of a $3$-dimensional cubic lattice with $\mathbb Z^3 \times \{x,y,z\}$ in that each edge is in bijection with a node of $\mathbb Z^3$ and one of the three positive directions $x,y,z$ in which the edge points from that node.  

\begin{definition} \cite[Definition 2.1]{BS} \label{def:Darboux}
 A map $f: \mathbb Z^3 \times \{x,y,z\} \rightarrow \mathbb R^3$ is a $3$-dimensional discrete Darboux map if for every face of the cubic lattice the images of its edges are collinear. In other words,
$$f_{i,j,k}^x, f_{i,j,k}^y, f_{i,j+1,k}^x, f_{i+1,j,k}^y \text{ are collinear, }$$
$$f_{i,j,k}^x, f_{i,j,k}^z, f_{i,j,k+1}^x, f_{i+1,j,k}^z \text{ are collinear, }$$
$$f_{i,j,k}^y, f_{i,j,k}^z, f_{i,j+1,k}^z, f_{i,j,k+1}^y \text{ are collinear. }$$
\end{definition}

\begin{remark}
 In Schief's definition \cite{Sch} the function takes values on faces of a cubic lattice, not edges. However, as Schief himself observes in loc. cit. the two are equivalent since one can consider a dual cubic lattice with vertices corresponding to elementary cubes of the original one. 
\end{remark}

One can think of discrete Darboux maps in a dynamical way in a similar fashion to $Q$-nets. Define the generation of an edge in $\mathbb Z^3 \times \{x,y,z\}$ as the sum of its three coordinates. Then it is easy to see, as pointed out in \cite[Exercice 2.8]{BS}, that each generation determines the next one uniquely. For example, $f_{i+1,j+1,k}^z$ is the intersection of the line connecting $f_{i+1,j,k}^y$ to $f_{i+1,j,k}^z$ with the line connecting $f_{i,j+1,k}^x$ to $f_{i,j+1,k}^z$. The fact that six points $f_{i+1,j,k}^y$, $f_{i+1,j,k}^z$, $f_{i,j+1,k}^x$, $f_{i,j+1,k}^z$, $f_{i,j,k+1}^x$, and $f_{i,j,k+1}^y$ lie in one plane is a necessary condition that is easily seen to self-propagate. 

The geometry of a discrete Darboux map is captured by the bipartite graph in Figure \ref{fig:k4}. Here on each edge of the lozenge tiling we place a white vertex signifying a point. To force the four points on the sides of a single lozenge to lie on one line we introduce two black vertices inside. It is clear that if the two triples of points lie on one line, then so do all four points. Figure \ref{fig:k4} should be compared for example with \cite[Figure 7]{KP}.

\begin{figure}[h!]
    \begin{center}
    \input{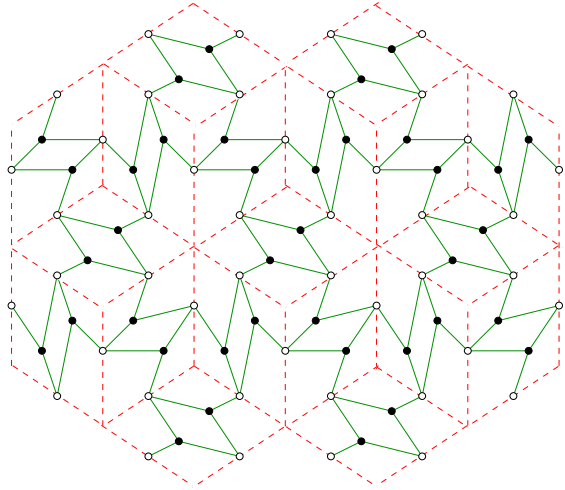}
    \end{center}
    \caption{The bipartite graph of a discrete Darboux map}
    \label{fig:k4}
\end{figure}

\begin{proposition}
 The sequence of square moves shown in Figure \ref{fig:k5} realizes geometrically a step of time evolution of the discrete Darboux map transitioning from vertices $G, H, K$ to vertices $L, M, N$ of the elementary hexahedron. 
\end{proposition}

\begin{figure}[h!]
    \begin{center}
    \input{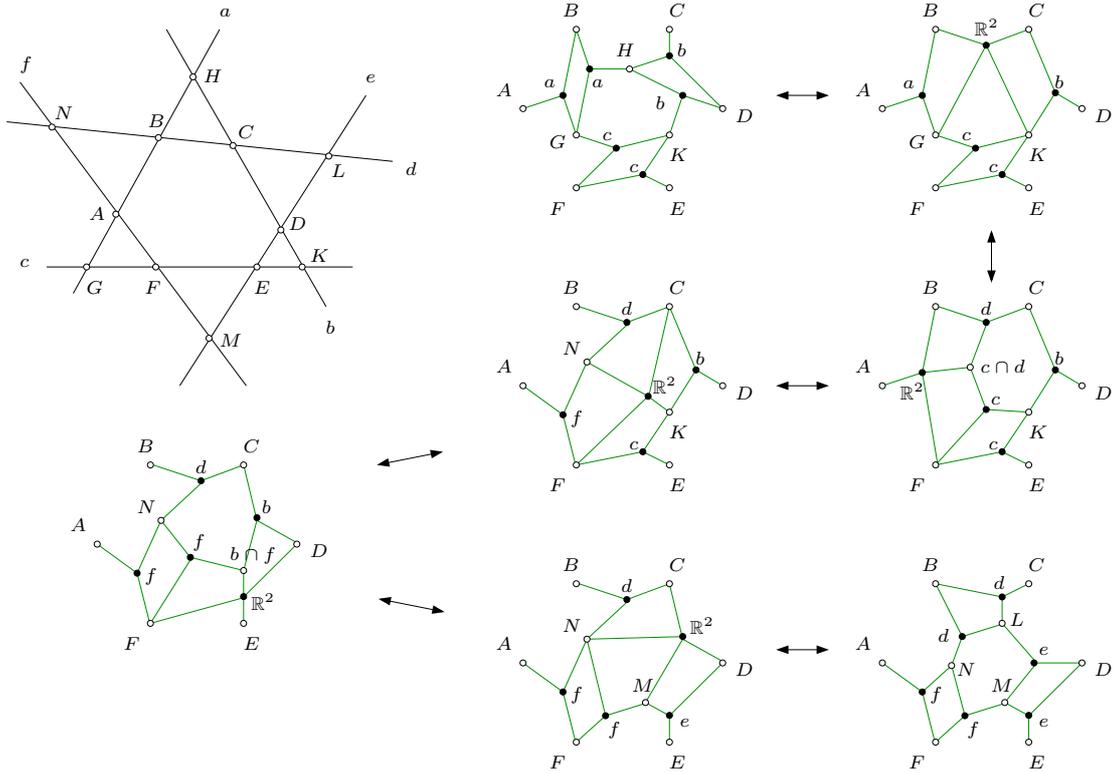}
    \end{center}
    \caption{Superurban renewal.}
    \label{fig:k5}
\end{figure}

\begin{proof}
 We verify the sequence of square moves using Proposition \ref{prop:evolution} on each step. 
\end{proof}

\begin{remark}
 This sequence of square moves has appeared in \cite[Figure 6]{KP}, without the current geometric interpretation, under the name of {\it {superurban renewal}}. 
\end{remark}

Proposition \ref{prop:faceweights} suggests we introduce the following variables, one for each region in Figure \ref{fig:k4}. For the variables associated with lozenges we get 
$$Y_{ijk}^{xy} = -[f_{ijk}^x, f_{i+1,j,k}^y, f_{i,j+1,k}^x, f_{ijk}^y]^{-1},$$ and similar formulas for other pairs of indices. 
The variables associated with vertices of lozenges come in three flavors, as there are three generations of them present in the picture. 
\begin{align*}
Y_{ijk}^{\textrm{in}} &= [f_{ijk}^x, f_{i+1,j,k}^y, f_{ijk}^y, f_{i,j+1,k}^z, f_{ijk}^z, f_{i,j,k+1}^x]^{-1}, \textrm{ for $i+j+k=t$} \\
Y_{ijk}^{\textrm{mid}} &= -[f_{ijk}^x, f_{i,j,k-1}^x, f_{i,j,k-1}^z, f_{i,j+1,k-1}^z, f_{ijk}^y, f_{i-1,j,k}^y, f_{i-1,j,k}^x, f_{i-1,j,k+1}^x, \\
& \quad \quad \quad f_{ijk}^z, f_{i,j-1,k}^z, f_{i,j-1,k}^y, f_{i+1,j-1,k}^y]^{-1}, \textrm{ for $i+j+k=t+1$} \\
Y_{ijk}^{\textrm{out}} &= [f_{i-1,j,k}^x, f_{i-1,j-1,k}^y, f_{i,j-1,k}^y, f_{i,j-1,k-1}^z, f_{i,j,k-1}^z, f_{i-1,j,k-1}^x]^{-1}, \textrm{ for $i+j+k=t+2$}.
\end{align*}
The quiver is shown in Figure \ref{fig:k6}. The $Y$-s evolve according to the $Y$-dynamics formulas of the associated cluster algebra. The formulas are too long to be written here.

\begin{figure}[h!]
    \begin{center}
    \input{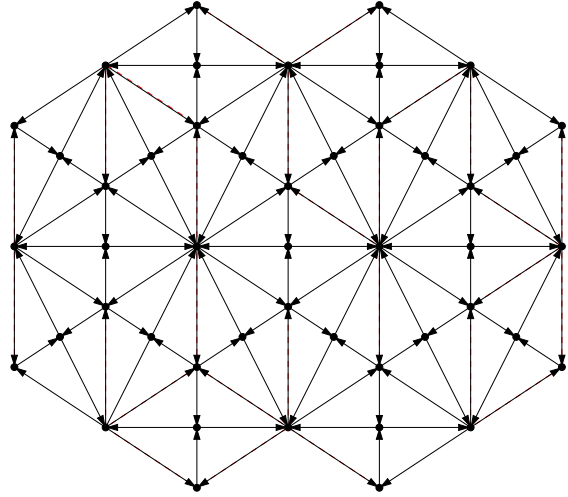}
    \end{center}
    \caption{Discrete Darboux map quiver.}
    \label{fig:k6}
\end{figure}
 
\begin{remark}
 The  $X$-variable dynamics associated with this quiver and sequence of mutations has appeared in \cite[Lemma 2.3]{KP}, see the formulas given there.
\end{remark}

\begin{remark}
The notions of $Q$-nets and discrete Darboux maps are related by projective duality.  As such, it is interesting that we get distinct quivers for these two systems. A general notion of projective duality for vector-relation configurations is developed in \cite{AGR}, capturing in particular the projective duality between $Q$-nets and discrete Darboux maps.
\end{remark}

\section{Geometric configurations for resistor networks and the Ising model} \label{sec:resist}
Goncharov and Kenyon \cite{GK} give a recipe to go from a resistor network given by an arbitrary weighted graph to a collection of edge weights on an associated bipartite graph.  There is an analogous recipe starting from the Ising model on a graph \cite{Dub,GalPyl,KLRR}.  In the case that the initial graph is a triangular grid, these constructions produce the same bipartite graphs discussed above for $Q$-nets (right of Figure \ref{fig:k1}) and discrete Darboux maps (Figure \ref{fig:k4}), respectively.  It turns out that the edge-weightings coming respectively from resistor networks and the Ising model represent very natural subfamilies of these geometric configurations, namely discrete Koenigs net and discrete CKP maps.  In this section we present these two examples of vector-relation configurations providing a link between physics and geometry.

A resistor network is a plane graph $G=(V,E)$ with each edge assigned a positive real weight interpreted as its conductance (i.e. reciprocal of resistance).  Suppose we draw the dual graph $G^* = (V^*,E^*)$ superimposed over a drawing of $G$.  The resulting picture can be interpreted as a bipartite graph $\Gamma$ whose white vertex set is $V \cup V^*$ and whose black vertices are the intersection points of dual edge pairs $e \in E, e^* \in E^*$.  Each edge of both $G$ and $G^*$ is subdivided in two, and all of the resulting half edges together comprise the edge set of $\Gamma$.  Assign each half of an edge in $E$ the same weight as the original edge, and assign each half of an edge in $E^*$ a weight of $1$.  Figure \ref{fig:resistor} illustrates the construction starting from a portion of the triangular grid graph $G$.

\begin{figure}
\centering
\includegraphics[height=2in]{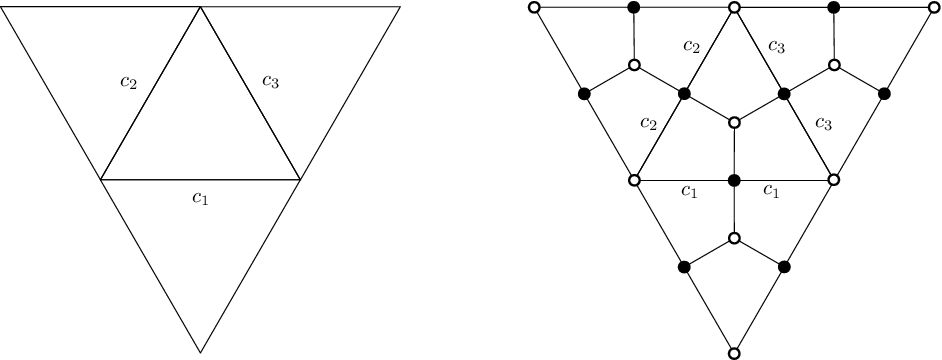}
\caption{Constructing a bipartite graph from a resistor network.  The unlabeled internal edges on the right have weight $1$.}
\label{fig:resistor}
\end{figure}

Let $G = (V,E)$ be the graph given by an infinite triangular grid and let $\Gamma$ be the associated bipartite graph.  Comparing Figures \ref{fig:k1} and \ref{fig:resistor}, we see that vector relation configurations on $\Gamma$ give three generations of a $Q$-net.  By Proposition \ref{prop:moves}, we can introduce signs to the weights coming from $G$ to get such a configuration.  

\begin{proposition}
Suppose $(Q_{i,j,k})$ is a $Q$-net constructed from a resistor network as above.  Then it is in fact a discrete Koenigs net, meaning that the points $Q_{i,j,k}$, $Q_{i+1,j+1,k}$, $Q_{i+1,j,k+1}$, and $Q_{i,j+1,k+1}$ are coplanar for all $i,j,k \in \mathbb{Z}$.
\end{proposition}

\begin{proof}
The graph in Figure \ref{fig:koenig} shows a small piece of $\Gamma$.  Consider each edge to have a negative sign if there is a stroke drawn through it and a positive sign otherwise.  This picture can be tiled to cover the plane and define signs on all edges of $\Gamma$.  The result satisfies the Kasteleyn condition: all faces are quadrilateral and each has either one or three negative edges on its boundary.  As such, the edge weights coming from $G$ multiplied by these signs give the relations of our vector-relation configuration.

The relations at the three black vertices in Figure \ref{fig:koenig} can now be read off as
\begin{equation} \label{eq:resistor}
\begin{split}
u+ c_1w_1 - u_1 - c_1w_2 &= 0 \\
u+ c_2w_2 - u_2 - c_2w_3 &= 0 \\
u+ c_3w_3 - u_3 - c_3w_1 &= 0
\end{split}
\end{equation}
Dividing relation $i$ by $c_i$ and summing we obtain
\begin{displaymath}
\left(\frac{1}{c_1} + \frac{1}{c_2} + \frac{1}{c_3}\right)u - \frac{1}{c_1}u_1 - \frac{1}{c_2}u_2 - \frac{1}{c_3}u_3 = 0.
\end{displaymath}
Therefore the projectivizations of $u,u_1,u_2,u_3$ all lie in a plane.  These four points are precisely $Q_{i,j,k}$, $Q_{i+1,j+1,k}$, $Q_{i+1,j,k+1}$, $Q_{i,j+1,k+1}$ for some $i,j,k$.  The relationship between the dynamics of resistor networks and the dimer model \cite{GK} guarantees that this property is preserved under the sequence of moves described in Section \ref{sec:Qnet}.  The equivalence of the coplanarity condition to other definitions of Koenigs nets is given in \cite[Theorem 2.29]{BS}.
\end{proof}

\begin{figure}
\centering
\includegraphics[height=2in]{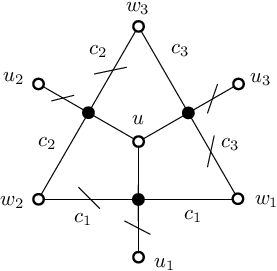}
\caption{A small part of a vector-relation configuration coming from a resistor network.}
\label{fig:koenig}
\end{figure}

\begin{remark}
The relations of \eqref{eq:resistor} can be rearranged to be instances of the discrete Moutard equation on the vectors at the white vertices.  This description gives another path via \cite[Theorem 2.32]{BS} to conclude that the $Q$-net is a discrete Koenigs net.
\end{remark}

\begin{remark}
Let $G=(V,E)$ be a resistor network with weight function $c$.  A \emph{discrete harmonic function} is a function $f$ on $V$, say with values in a vector space, satisfying the condition
\begin{displaymath}
\sum_{v'} c(\overline{vv'})(f(v)-f(v')) = 0
\end{displaymath}
for all $v \in V$.  The harmonic condition is equivalent to the existence of a second function $g$ on $V^*$ satisfying
\begin{displaymath}
g(w)-g(w') = c(\overline{vv'})(f(v)-f(v'))
\end{displaymath}
for each dual edge pair $\overline{vv'} \in E$, $\overline{ww'} \in E^*$ (a convention needs to be fixed for the direction of the crossing of the edges), see e.g. \cite[Section 6]{KLRR}.  If $G$ is the hexagonal grid, so $G^*$ is the dual triangular grid, the above precisely means that $f$ and $g$ together define valid vectors for the associated vector-relation configuration on $G$.  The picture is as in Figure \ref{fig:koenig} except with the non-trivial weights $c_i$ moved to the other half of the edges.  Some care with signs would be needed to extend this idea to other graphs.
\end{remark}

We next consider the Ising model.  We follow the approach of Galashin and Pylyavskyy \cite{GalPyl}.  Figure \ref{fig:Ising} gives an example of a bipartite graph arising from an Ising network.  Each unlabeled edge has weight $1$ and the $s_i,c_i$ are certain positive reals satisfying $c_i^2 + s_i^2 = 1$.  Roughly speaking, the construction replaces each edge of the original graph with a copy of the $Gr_{2,4}$ plabic graph of Figure \ref{fig:gr24}.  In the case of Figure \ref{fig:Ising}, the original graph consisted of a single triangle whose $i$th edge passes through both new edges marked $s_i$.

\begin{figure}
\centering
\includegraphics[height=2in]{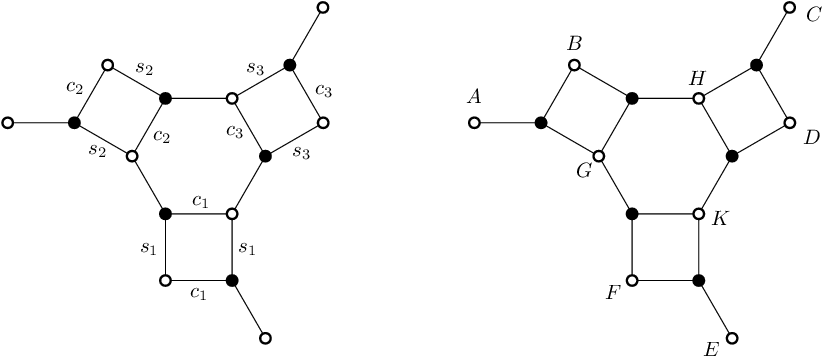}
\caption{A bipartite graph coming from the Ising model.}
\label{fig:Ising}
\end{figure}

\begin{proposition}
Consider a circuit configuration in $\mathbb{R}^3$ of the graph in Figure \ref{fig:Ising} with $A,B,\ldots \in \mathbb{P}^2$ the projectivizations of the points as indicated.  Then the six points $A,B,C,D,E,F$ lie on a conic.
\end{proposition}

\begin{proof}
Let $Y$ be the face weight of the hexagonal face and let $Y_i$ for $i=1,2,3$ be the weights of the quadrilateral faces.  On the one hand, these can be computed in terms of the edge weights
\begin{align*}
Y &= c_1c_2c_3 \\
Y_i &= \frac{s_i^2}{c_i^2} = \frac{1}{c_i^2} - 1
\end{align*}
from which we get
\begin{displaymath}
Y^2(1+Y_1)(1+Y_2)(1+Y_3) = 1.
\end{displaymath} 
Meanwhile, by Proposition \ref{prop:faceweights} 
\begin{align*}
Y &= [G,B,H,D,K,F]^{-1} \\
Y_1 = -[F,G,K,E]^{-1} \quad \quad 1+Y_1 &= [F,K,E,G] \\
Y_2 = -[G,A,B,H]^{-1} \quad \quad 1+Y_2 &= [G,B,H,A] \\
Y_3 = -[H,C,D,K]^{-1} \quad \quad 1+Y_3 &= [H,D,K,C] \\
\end{align*}
so
\begin{displaymath}
[G,B,H,D,K,F]^2 = [F,K,E,G][G,B,H,A][H,D,K,C].
\end{displaymath}
Every factor occurring in $[G,B,H,D,K,F]$ appears once in the right hand side, and when canceled out, what remains is a triple ratio $[G,E,K,C,H,A]$.  Therefore
\begin{displaymath}
[G,B,H,D,K,F] = [G,E,K,C,H,A].
\end{displaymath}
The relative position of the points is as in the top left of Figure \ref{fig:k5} and it follows from Carnot's Theorem \cite{Carnot} that $A,B,C,D,E,F$ lie on a conic.
\end{proof}

Now suppose we begin with an infinite triangular grid.  The associated bipartite graph is the one in Figure \ref{fig:k4} whose configurations correspond to discrete Darboux maps.  What we have shown is that, in the notation of Definition \ref{def:Darboux}, a Darboux map arising from the Ising model has the property that for all $i,j,k$ the points
\begin{displaymath}
f_{i+1,j,k}^y, f_{i+1,j,k}^z, f_{i,j+1,k}^x, f_{i,j+1,k}^z, f_{i,j,k+1}^x, f_{i,j,k+1}^y
\end{displaymath}
lie on a conic.  This reduction of Darboux maps has been studied by Schief under the name \emph{discrete CKP maps} \cite{Sch}.

\begin{proposition}
Any discrete Darboux map arising from the Ising model on an infinite triangular grid is in fact a discrete CKP map.
\end{proposition}

\section{Configurations on plabic graphs} \label{sec:plabic}
We now consider the plabic graph case, including the main definition in this setting (Section \ref{sec:restriction}) and the proof of Theorem \ref{thm:plabic} (Sections \ref{sec:target}--\ref{sec:unique}).  An alternate point of view for this story in terms of the boundary measurement map will be given in Section \ref{sec:measurement}.  For a quicker summary of how these pieces fit together see Remark~\ref{rem:measurement}.

\subsection{Background on positroid varieties} \label{sec:background}
The proof of Theorem \ref{thm:plabic} utilizes a significant amount of the theory of positroid varieties.  We begin by reviewing the relevant material, generally following \cite{MulSpe} and \cite{LamCDM}.

A \emph{plabic graph} is a finite planar graph $G=(B \cup W, E)$ embedded in a disk with the vertices all colored black or white.  We assume throughout that $G$ is in fact bipartite, that all of its boundary vertices are colored white, and that each boundary vertex has degree $1$ or $0$.  An \emph{almost perfect matching} of $G$ is a matching that uses all internal vertices (and some boundary vertices).  Assume always that $G$ has at least one almost perfect matching.

\begin{remark}
The most common formulation these days \cite{LamCDM, MulSpe} is to assume that $G$ is bipartite with the boundary vertices being uncolored and all having degree $1$.  Starting from such a graph, one can use degree $2$ vertex addition where needed on boundary edges to get each boundary vertex adjacent to a black vertex. At that point, boundary vertices can be colored white to adhere to our conventions.  The exception is if the original graph has a degree $1$ white vertex attached to the boundary.  The above procedure would produce a graph that is not reduced, a condition we will eventually require.  For us, an isolated boundary vertex models this situation.  
\end{remark}

Fix for the moment a plabic graph $G=(B \cup W, E)$.  Let $M = |B|$, $N=|W|$, and let $n$ be the number of boundary vertices.  As all boundary vertices are white that leaves $N-n$ internal white vertices.  Number the elements of $B$ and $W$ respectively $1$ through $M$ and $1$ through $N$  in such a way that the boundary (white) vertices are numbered $1$ through $n$ in clockwise order. Let $k = N-M$.  Each almost perfect matching uses all $M$ black vertices and all $N-n$ internal white vertices.  As such it must use $M-(N-n) = n-k$ boundary vertices, from which we conclude $0 \leq k \leq n$, with the interesting case being $0 < k < n$.

The totally nonnegative Grassmannian is the set of $A \in Gr(k,n)$ for which the Pl\"ucker coordinate $\Delta_J(A)$ is real and nonnegative for all $J$.  The \emph{matroid} of any $A \in Gr(k,n)$ is 
\begin{displaymath}
\mathcal{M} = \{J : \Delta_J(A) \neq 0\}.
\end{displaymath}  
A \emph{positroid} is a set of $k$-element subsets of $\{1,\ldots, n\}$ that arises as the matroid of a point in the totally nonnegative Grassmannian.  We also denote a positroid by $\mathcal{M}$ even though this is a more restrictive notion than a matroid.

Let $\mathcal{M}$ be a positroid.  For $j=1,\ldots,n$, consider the column order $j < j+1 < \ldots < n < 1 < \ldots < j-1$.  Let $I_j$ be the lexicographically minimal element of $\mathcal{M}$ relative to this order.  The collection of sets $(I_1,\ldots, I_n)$ is called the \emph{Grassmann necklace} of $\mathcal{M}$.  The positroids index a decomposition of the complex Grassmannian by \emph{open positroid varieties} $\Pi_{\mathcal{M}}^{\circ}$, defined as intersections of cyclic shifts of Schubert cells encoded by $(I_1,\ldots, I_n)$.  The \emph{positroid variety} $\Pi_{\mathcal{M}}$ is defined to be the Zariski closure of $\Pi_{\mathcal{M}}^{\circ}$.  In order to give quicker definitions, we fall back on the literature.

\begin{theorem}[Knutson--Lam--Speyer \cite{KLS}]
The positroid variety $\Pi_{\mathcal{M}}$ is a closed irreducible variety defined in the Grassmannian by
\begin{displaymath}
\Pi_{\mathcal{M}} = \{A \in Gr_{k,n} : \Delta_J(A) = 0 \textrm{ for all } J \notin \mathcal{M}\}. 
\end{displaymath}
\end{theorem}
Taking this result as given we can define $\Pi_{\mathcal{M}}^{\circ}$ as the set of $A \in \Pi_{\mathcal{M}}$ whose Pl\"ucker coordinates $\Delta_{I_j}(A)$ coming from the Grassmann necklace are all nonzero.

Let $G=(B \cup W, E)$ be a plabic graph.  Following our conventions, all boundary vertices are white.  An almost perfect matching is a matching in $G$ that uses all internal vertices.  Hence it is a matching of $B$ with $W \setminus J$ for some $J \subseteq \{1,\ldots, n\}$ (identified with the boundary vertices) satisfying $|J| = k$.  The \emph{positroid} of $G$, denoted $\mathcal{M}_G$ is the set of $J$ that arise this way as the unused vertices of an almost perfect matching.

The \emph{boundary measurement map} is a function that takes as input a set of nonzero edge weights on $G$ and outputs a point $A \in \Pi_{\mathcal{M}}^{\circ}$, where $\mathcal{M} = \mathcal{M}_G$.  If $\wt: E \to \mathbb{C}^*$ is the weight function then $A$ is defined by its Pl\"ucker coordinates via
\begin{equation} \label{eq:measurement}
\Delta_J(A) = \sum_{\pi} \prod_{e \in \pi} \wt(e) 
\end{equation}
where the sum is over matchings $\pi$ of $B$ with $W \setminus J$.  The result is unchanged by gauge transformations at internal vertices.  The boundary measurement map plays a key role in the study of the nonnegative Grassmannian as it proves that the individual strata therein are cells.

The situation is more complicated in the complex case as the boundary measurement map is not surjective.  Its image, which will play a key role for us, was identified by Muller and Speyer \cite{MulSpe}.  First, they define a remarkable isomorphism $\tau: \Pi_{\mathcal{M}}^{\circ} \to \Pi_{\mathcal{M}}^{\circ}$ called the \emph{right twist}.  Suppose $A = [v_1 \cdots v_n]$ and $\tau(A) = [v_1' \cdots v_n']$.  We do not give the full definition of the twist, but instead state a key property (that defines the $v_j'$ up to scale).  Specifically for each $j$, $v_j'$ is orthogonal to $v_i$ for all $i \in I_j \setminus \{j\}$.

The last piece of technology we need, both in relation to the boundary measurement map and for other purposes, is the notion of zigzag paths in $G$.  A \emph{zigzag path} is a path of $G$ that turns maximally left (respectively right) at each white (respectively black) vertex and either starts and ends at the boundary or is an internal cycle.  Each directed edge can be extended to a zigzag path, so there are two zigzag paths through each edge.  Define an \emph{intersection} of two zigzags to be such an edge that they traverse in opposite directions.  Say that $G$ is \emph{reduced} if
\begin{itemize}
\item each zigzag path starts and ends at the boundary,
\item each zigzag of length greater than two has no self intersections
\item no pair of distinct zigzags have a pair of intersections that they encounter in the same order.
\end{itemize}

If $G$ is reduced then there are exactly $n$ zigzags, one starting at each boundary vertex.  Call $j$ the number of the zigzag starting at vertex $j$.  A zigzag that does not self intersect divides the disk into two regions.  For $F$ a face of $G$, let $S_F$ denote the set of $j$ for which $F$ lies to the left of zigzag number $j$.  There are two corner cases.  If $j$ is attached to a degree $1$ black vertex $b$ then zigzag $j$ goes from $j$ to $b$ and back to $j$.  In this event all faces are considered to be to the right of the zigzag.  On the other hand, if $j$ is an isolated boundary vertex then zigzag $j$ is an empty path that all faces are considered to lie to the left of.  With these conventions one can show that all $S_F$ have size $k$.

\begin{theorem}[Muller-Speyer \cite{MulSpe}, Theorem 7.1] \label{thm:mulspe}
The image of the boundary measurement map is the set of $A \in \Pi_{\mathcal{M}}^{\circ}$ whose twist $A' = \tau(A)$ satisfies 
\begin{displaymath}
\Delta_{S_F}(A') \neq 0
\end{displaymath}
for all faces $F$ of $G$.  This set is dense in $\Pi_{\mathcal{M}}^{\circ}$ and in fact the coordinates $\Delta_{S_F}(A')$ give it the structure of an algebraic torus.
\end{theorem}

\subsection{The boundary restriction map} \label{sec:restriction}
Theorem \ref{thm:plabic} should be understood with respect to a modified definition of vector-relation configurations specifically catered to plabic graphs.  In this section, we first provide this definition, then we reformulate Theorem \ref{thm:plabic} to clarify the connection with the various notions described in Section \ref{sec:background}.

Let $G= (B \cup W, E)$ be a plabic graph with all the notation of Section \ref{sec:background}.  In defining a vector-relation configuration on $G$, we will see the natural ambient dimension is $k=N-M$.  As such, we simply fix as our vector space $V = \mathbb{C}^k$.  It is also natural to allow boundary vectors to be zero, and to add some genericity assumptions.  In the following, let $K_{ij}$ denote the coefficient of the vector $v_j$ in relation $R_i$ where $1 \leq i \leq M$ and $1 \leq j \leq N$.

\begin{definition} \label{def:vr-plabic}
A \emph{vector-relation configuration} on a plabic graph $G$ is a choice of vector $v_w \in V=\mathbb{C}^k$ for each $w \in W$ and a non-trivial linear relation $R_b$ among the neighboring vectors of each $b \in B$ such that
\begin{itemize}
\item the vector $v_w$ at each internal white vertex $w$ is nonzero,
\item the boundary vectors $v_1,\ldots, v_n$ span $V$, and
\item the $M \times N$ matrix $K = (K_{ij})$ is full rank.
\end{itemize}
Two configurations are called \emph{gauge equivalent} if they are related by a sequence of gauge transformations, in the sense of Definition \ref{def:gauge}, at internal vertices.
\end{definition}

Let $\mathcal{C}_G$ denote the space of gauge equivalence classes of vector-relation configurations on $G$ modulo the action of $GL_k(\mathbb{C})$.  If $(\v, \Rvec) \in \mathcal{C}_G$ then by assumption $v_1,\ldots, v_n$ span $V = \mathbb{C}^k$.  The $v_i$ are defined up to a common change of basis so $A = [v_1 \cdots v_n]$ is a well-defined point of $Gr_{k,n}$.  We use $\Phi$ to denote the map $\Phi:\mathcal{C}_G \to Gr_{k,n}$ taking $(\v, \Rvec)$ to $A$, and we call $\Phi$ the \emph{boundary restriction map}.  

In this language, Theorem \ref{thm:plabic} asserts that $\Phi$ maps $\mathcal{C}_G$ into $\Pi_{\mathcal{M}}$ and that generic points in this positroid variety have unique preimages.  We next identify a set $T_G \subset \Pi^{\circ}_{\mathcal{M}}$ whose elements are sufficiently generic for this purpose.  Specifically, let
\begin{equation} \label{eq:generic}
T_G = \{A \in \Pi^{\circ}_{\mathcal{M}} : \Delta_{S(F)}(A') \neq 0 \textrm{ for all faces } F \textrm{ of } G\}
\end{equation}
where $A'=\tau(A)$ is the result of applying the right twist to $A$.

\begin{remark} \label{rem:measurement}
Note that $T_G$ is precisely the image of the boundary measurement map, as demonstrated by Muller and Speyer \cite{MulSpe} and reviewed in Theorem \ref{thm:mulspe}.  In fact, the boundary restriction map and the boundary measurement map are very closely related, a connection we explore in Section~\ref{sec:measurement}.  Once that is done many of our results follow from analogous ones in \cite{MulSpe}.  We focus first on presenting a derivation of Theorem~\ref{thm:plabic} which uses neither the connection between the boundary restriction map and the boundary measurement map nor Theorem \ref{thm:mulspe}. We will however make extensive use of background material developed in \cite{MulSpe}, specifically in Sections 2 -- 6 and Appendix B of that paper. After proving Theorem~\ref{thm:plabic}, we provide a proof of Theorem~\ref{thm:main} (a stronger version of Theorem~\ref{thm:plabic}), which does make use of Theorem \ref{thm:mulspe}.
\end{remark}

As an example, we prove without appealing to Theorem \ref{thm:mulspe} that $T_G \subseteq \Pi_{\mathcal{M}}$ is dense.  It suffices to show $T_G$ is dense in $\Pi^{\circ}_{\mathcal{M}}$ since the latter is dense in $\Pi_{\mathcal{M}}$.  By \eqref{eq:generic}, we have a collection of open conditions and it remains to show that each is satisfiable, i.e. that no $\Delta_{S(F)}(A')$ is uniformly zero on $\Pi^{\circ}_{\mathcal{M}}$.  Let $F$ be a face of $G$.  By \cite[Theorem 5.3]{MulSpe}, there is an almost perfect matching of $G$ avoiding the set $S(F)$ of boundary vertices.  Hence, applying the boundary measurement map (see \eqref{eq:measurement}) to any choice of positive edge weights gives a point $A' \in \Pi^{\circ}_{\mathcal{M}}$ with $\Delta_{S(F)}(A') \neq 0$.  By \cite[Corollary 6.8]{MulSpe} the twist is invertible on $\Pi^{\circ}_{\mathcal{M}}$ and we can recover $A$.

\subsection{Identifying the target} \label{sec:target}
We begin with the first part of Theorem \ref{thm:plabic}, namely that the positroid variety $\Pi_{\mathcal{M}}$ can be taken to be the target of the boundary restriction map $\Phi$.

\begin{lemma} \label{lem:rowspace}
Let $\v \in \mathcal{C}_G$.  There is a surjective linear map $\phi: \mathbb{C}^N \to V$ with kernel equal to the row span $\row(K)$ such that each $v_w \in V$ is mapped to by the corresponding coordinate vector $e_w \in \mathbb{C}^N$.
\end{lemma}

\begin{proof}
We can define a linear map $\phi$ via $\phi(e_w) = v_w$ for all $w \in W$.  As $v_1,\ldots, v_n$ span $V$, this map is surjective.  Any given row of $K$ is indexed by some $b \in B$, and equals $\sum_w K_{bw}e_w$.  As such it gets mapped to $\sum_w K_{bw}v_w$ which equals zero by relation $R_b$.  So $\row(K) \subseteq \ker(\phi)$.  But 
\begin{displaymath}
\dim \row(K) = M
\end{displaymath}
since $K$ is full rank and
\begin{displaymath}
\dim \ker(\phi) = N - \dim(V) = N-k = M
\end{displaymath}
since $\phi$ is surjective so $\ker(\phi) = \row(K)$.
\end{proof}

The previous establishes that a configuration in this setting is completely determined by $K$.  More precisely, say two vector-relation configurations on the same graph are isomorphic if
\begin{itemize}
\item there is an isomorphism of their ambient spaces that identifies corresponding (at the same white vertex) vectors, and
\item the corresponding (at the same black vertex) relations are equal.
\end{itemize}
Then, $K$ determines a configuration in the ambient space $\mathbb{C}^N / \row(K)$ as above whose vectors are projections of the coordinate vectors and which is isomorphic to any other configuration giving rise to $K$.

\begin{lemma} \label{lem:determinant}
Let $\v \in \mathcal{C}_G$ and let $S = \{w_1,\ldots, w_k\} \subseteq W$.  Then
\begin{displaymath}
\det [v_{w_1} \cdots v_{w_k}] = \pm \lambda \Delta_{W \setminus S}(K)
\end{displaymath}
where $\lambda$ is a nonzero scalar not depending on $S$ and $\Delta_J$ denotes the determinant of a submatrix consisting of all rows and a specified set $J$ of columns of a matrix. 
\end{lemma}

\begin{proof}
By Lemma \ref{lem:rowspace}, $\v$ is isomorphic to the configuration of the projections of coordinate vectors in $U = \mathbb{C}^N / \row(K)$.  Viewing elements of $U$ as equivalence classes of row vectors in $\mathbb{C}^N$, there is a well-defined, multilinear, alternating map
\begin{displaymath}
(u_1, \ldots, u_k) \in U^k \mapsto \det \left[\begin{array}{c} u_1 \\ \vdots \\ u_k \\ K \end{array}\right] \in \mathbb{C}.
\end{displaymath}
Applied to $e_{w_1}, \ldots, e_{w_k}$, the result equals $\pm \Delta_{W \setminus S}(K)$ (if $S$ is in increasing order the sign is determined by the parity of $(w_1-1) + \ldots + (w_k-k)$).  Pulling back via the isomorphism, this map corresponds to some nonzero multiple of the determinant and in particular gives the desired formula for $\det [v_{w_1} \cdots v_{w_k}]$.
\end{proof}

\begin{corollary} \label{cor:KToA}
Let $\v \in \mathcal{C}_G$ and $A = \Phi(\v)$.  Let $J = \{j_1,\ldots, j_k\}$ with $1 \leq j_1 < \ldots < j_k \leq n$.  Then the Pl\"ucker coordinates of $A$ are
\begin{equation} \label{eq:KToA}
\Delta_J(A) = \pm \Delta_{W \setminus J}(K)
\end{equation}
with the sign determined by the parity of $(j_1-1) + \ldots + (j_k-k)$.
\end{corollary}

\begin{proof}
A representing matrix for $A$ is $[v_1 \cdots v_n]$, and we can compute its minors using Lemma \ref{lem:determinant}.  The Pl\"ucker coordinates are only defined up to multiplication by a common constant so we can ignore the $\lambda$'s.
\end{proof}

Unfolding \eqref{eq:KToA}, we have
\begin{displaymath}
\Delta_J(A) = \pm \sum_{f} \sgn(f) \prod_{b \in B} K_{b,f(b)}
\end{displaymath}
where the sum is over bijections $f$ from $B$ to $W \setminus J$ and $\sgn(f)$ is defined by thinking of $f$ as a permutation (we assumed linear orders on $B$ and $W$, and the latter restricts to a linear order on $W \setminus J$). In fact $K_{bw} = 0$ unless $\overline{bw}$ is an edge, so we only get a nonzero term if the set of $\overline{bf(b)}$ forms an almost perfect matching of $G$ avoiding the vertex set $J$.  So we can rewrite the formula as
\begin{equation} \label{eq:restriction}
\Delta_J(A) = \pm \sum_{\pi} \sgn(\pi) \prod_{\overline{bw} \in \pi} K_{bw},
\end{equation}
the sum being over such almost perfect matchings $\pi$.

\begin{proposition} \label{prop:target}
If $\v \in \mathcal{C}_G$ then $\Phi(\v) \in \Pi_{\mathcal{M}}$.
\end{proposition}

\begin{proof}
Let $J \subseteq \{1,\ldots, n\}$ with $|J| = k$ and suppose $J \notin \mathcal{M}$.  By definition of $\mathcal{M}$ there is no almost perfect matching of $G$ avoiding $J$.  Therefore the sum in \eqref{eq:restriction} is empty and we get $\Delta_J(\Phi(v)) = 0$.  So $\Phi(v)$ satisfies the defining equations of $\Pi_{\mathcal{M}}$.
\end{proof}

The last result identifies linear dependent sets of size $k$ among the boundary vectors.  The result generalizes easily.  

\begin{proposition} \label{prop:matroid}
Let $\v \in \mathcal{C}_G$ and let $S \subseteq W$ be any set of white vertices.  Suppose there is no matching of $B$ with a subset of $W$ disjoint from $S$.  Then the vectors $v_w$ for $w \in S$ are linearly dependent.
\end{proposition}

\begin{proof}
First suppose $|S| = k$.  Then the $v_w$ for $w \in S$ form a square matrix whose determinant can be calculated using Lemma \ref{lem:determinant}.  There is no matching of $B$ with $W \setminus S$ so the right hand side is zero and the vectors are dependent.  If $|S| < k$ then we can augment $S$ arbitrarily to get a set of size $k$ satisfying the same hypotheses and hence corresponding to a dependent set.  In other words $\{v_w : w \in S\}$ cannot be extended in the configuration to a basis of $V$.  All vectors together span $V$ so it follows that the set is dependent.
\end{proof}

\begin{remark}
Restricting to the $|S| = k$ case, one might hope for the stronger statement that $\{v_w : w \in S\}$ is a basis if and only if there is a matching of $B$ with $W \setminus S$.  The if direction only holds for generic $\v \in \mathcal{C}_G$.  In the generic case, the matroid of the vectors of $\v$ is dual to the so-called \emph{transversal matroid} of the bipartite graph $G$.  This result is very similar to one of Lindstr\"om \cite{Lin}.  The similarity comes as no surprise as Lindstr\"om's famed lemma, which he introduced in that paper, is an essential ingredient in the boundary measurement map.
\end{remark}

\subsection{The reconstruction map} 
In this subsection, we begin to prove the second part of Theorem \ref{thm:plabic}.  Specifically, we define a map $\Psi$ on a dense subset of $\Pi_{\mathcal{M}}$ which will turn out to be a right inverse of $\Phi$.  As $\Psi$ has the effect of reconstructing the entire configuration from just the boundary vectors, we term it the \emph{reconstruction map}.  We temporarily add an assumption on $G$ that there is no isolated boundary vertex and no boundary vertex attached to a vertex of degree $1$.  Since $G$ is reduced this condition is equivalent to saying $\mathcal{M}$ has a basis containing $j$ and one excluding $j$ for each $j=1,\ldots, n$.  It follows that $j \in I_j$ and $j \notin I_{j+1}$.

It is convenient at this point to introduce an alternate representation of zigzag paths known as strands.  A strand is obtained from a zigzag by taking each turn of the zigzag and replacing it with an arc connecting the midpoints of two edges involved.  Based on the zigzag rules, the arc appears to go clockwise around a white vertex and counterclockwise around a black vertex.  The \emph{strand} is obtained by combining all arcs of a zigzag as well as small pieces at the beginning and end to connect it to the boundary of the disk.  The strands together form an \emph{alternating strand diagram}, one example of which is given in Figure \ref{fig:strand}.  Note that strand number $i$ begins slightly clockwise relative to boundary vertex $i$.  An intersection of zigzags as defined previously translates to an intersection in the usual sense of strands.

\begin{figure}
\centering
\includegraphics[height=2.5in]{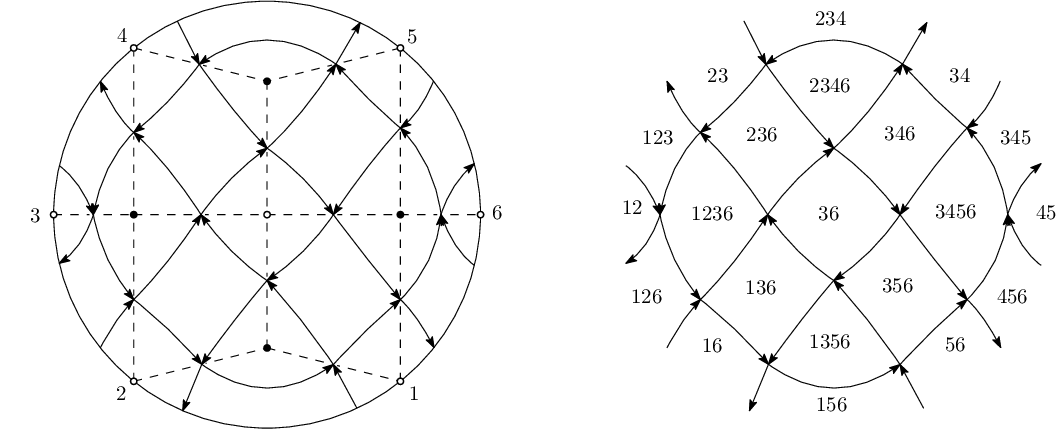}
\caption{The alternating strand diagram for a plabic graph (left) and the associated labeling by sets of the faces and vertices of the graph (right).}
\label{fig:strand}
\end{figure}

Each region of an alternating strand diagram has boundary oriented clockwise, counterclockwise, or in an alternating manner, and the region corresponds respectively to a white vertex, black vertex, or face of $G$.  Use the notations $S_w$, $S_b$, and $S_F$ to denote the set of strands that the region in the strand diagram associated to $w$, $b$, or $F$ 
lies to the left of.  For $F$ a face, this definition agrees with the previously given zigzag one.

\begin{remark}
To avoid strands altogether, one could define $S_w$ and $S_b$ in terms of the face labels via $S_w = \cap_F S_F$ and $S_b = \cup_F S_F$ where both formulas range over all faces containing the vertex in question.
\end{remark}

\begin{proposition} \label{prop:labels}
Let $F$ be a face of $G$, $b \in B$, and $w \in W$.
\begin{itemize}
\item $|S_F| = k$, $|S_w| = k-1$, and $|S_b| = k+1$.
\item If $b$ and $w$ are on the boundary of $F$ then $S_w \subseteq S_F \subseteq S_b$.
\end{itemize}
\end{proposition}

\begin{proof}
It is standard that each $S_F$ has size $k$.  If $w$ is a white vertex of $F$ then there is a zigzag through $w$ that enters and exits along the boundary of $F$ and turns left at $w$.  The corresponding strand divides the regions corresponding to $F$ and $w$ with the region corresponding to $F$ on the left.  Therefore $S_w$ equals $S_F$ less that one strand.  In particular $S_w \subseteq S_F$ with $|S_w| = k-1$.  A similar argument applies to black vertices.
\end{proof}

We begin to construct the inverse of the boundary restriction map on $T_G$ as defined in \eqref{eq:generic}.  Fix $A \in T_G \subseteq \Pi_{\mathcal{M}}$ and in fact fix a particular matrix representative so that the columns $v_1,\ldots, v_n$ of $A$ all live in $V=\mathbb{C}^k$.  Let $H_j \subseteq V$ denote the linear span of $\{v_i : i \in I_j \setminus \{j\}\}$.  For each $w \in W$, define
\begin{equation} \label{eq:reconstruct}
L_w = \bigcap_{j \in S_w} H_j.
\end{equation}
Recall in the following that $v_j'$ denotes column $j$ of the right twist of $A$.

\begin{lemma} \ 
\begin{enumerate}
\item Each $H_j$ is a hyperplane with orthogonal complement spanned by $v_j'$.
\item The $k$ hyperplanes of the set $\{H_j: j \in S_F\}$ are in general position for each face $F$.
\item Each $L_w$ is a line.
\end{enumerate}
\end{lemma}

\begin{proof} \ 
\begin{enumerate}
\item Since $A \in T_G \subseteq \Pi_{\mathcal{M}}^{\circ}$ we know that $\Delta_{I_j}(A) \neq 0$ so the $v_i$ with $i \in I_j$ form a basis of $V$.  We know that $j \in I_j$ so $H_j$ is a span of all but one of these vectors and is hence a hyperplane.  The twist is defined in such a way that $v_j'$ is nonzero and orthogonal to each $v_i$ for $i \in I_j \setminus \{j\}$, so $v_j'$ is the orthogonal complement of $H_j$.
\item It is equivalent to say that the orthogonal vectors $v_j'$ for $j \in S_F$ form a basis of $V$.  This holds true since, by definition of $T_G$ in~\eqref{eq:generic}, $\Delta_{S_F}(A')\neq0$.
\item
By Proposition~\ref{prop:labels}, for every $w$, $|S_w| = k-1$.  So $L_w$ is an intersection in $V = \mathbb{C}^k$ of $k-1$ hyperplanes in general position and is hence a line.
\end{enumerate}
\end{proof}

\begin{proposition} \label{prop:circuit}
Let $b \in B$ and choose nonzero vectors $v_w \in L_w$ for each neighbor $w$ of $b$.  Then these $v_w$ satisfy a unique linear relation up to scale, and this relation has all coefficients nonzero.
\end{proposition}

\begin{proof}
Suppose $b$ has degree $d$ and let $j_1,\ldots,j_d$ be the numbers of the strands around $b$ in counterclockwise order.  For $i=1,\ldots d$ there is a face $F_i$ separated from $b$ by strand $j_i$.  There is an edge shared by $F_{i-1}$ and $F_i$ (indices modulo $d$) whose endpoints are $b$ and some $w_i$.  Then $w_1,\ldots, w_d$ are the neighbors of $b$ and we have
\begin{itemize}
\item $S_{F_i} = S_b \setminus \{j_i\}$,
\item $S_{w_i} = S_b \setminus \{j_{i-1},j_i\}$.
\end{itemize}

Let $S = S_b \setminus \{j_1,\ldots, j_d\}$.  Then for each $i$, $S \subseteq S_{w_i}$ so
\begin{displaymath}
v_{w_i} \in L_{w_i} \subseteq \cap_{a \in S} H_a.
\end{displaymath}
Also, $S \subseteq S_{F_1}$ so the hyperplanes in this intersection are in general position.  As $|S| = k+1 - d$ we have that $\cap_{a \in S} H_a$ has dimension $d-1$.  Therefore the $d$ vectors $v_{w_i}$ in this space must satisfy a relation.

Now suppose $c_1 v_{w_1} + \ldots + c_d v_{w_d} = 0$ is a non-trivial relation.  Note that $j_1 \in S_{w_i}$ for all $i=3,\ldots, d$, so $v_{w_i} \in H_{j_1}$ for these $i$.  On the other hand, $v_{w_1} \notin H_{j_1}$ because otherwise we would have $v_{w_1} \in H_a$ for all $a \in S_{w_1} \cup \{j_1\} = S_{F_d}$ which would imply $v_{w_1} = 0$.  A similar argument shows $v_{w_2} \notin H_{j_1}$.  Therefore, we can apply a linear functional vanishing at $H_{j_1}$ (e.g. the dot product with $v_{j_1}'$) to the above relation and precisely the first two terms survive.  It follows that $c_1$ and $c_2$ are either both zero or both nonzero and have a prescribed ratio.  The same is true by symmetry for each pair of consecutive coefficients.  We cannot have all $c_i=0$ so the $c_i$ are all nonzero and are unique up to multiplication by a common factor.
\end{proof}

\begin{proposition} \label{prop:existence}
Let $A \in T_G$.  Then there exists a unique configuration $(\v,\Rvec) \in \mathcal{C}_G$ such that $\Phi(\v,\Rvec) = A$ and $v_w \in L_w$ for all $w \in W$.  This configuration has the property that the set of vectors neighboring each black vertex is a circuit.
\end{proposition}

\begin{proof}
Let $v_j$ equal column $j$ of $A$.  First we show $v_j \in L_j$ holds for these eventual boundary vectors.  Consider the boundary face $F$ of $G$ containing the boundary segment between $j$ and $j+1$.  By \cite[Proposition 4.3]{MulSpe}, $S_F$ equals the set $\tilde{I}_j$ in the so-called reverse Grassmann necklace of $\mathcal{M}$.  The strand separating face $F$ from white vertex $j$ is in fact strand number $j$ so $S_j = \tilde{I}_j \setminus \{j\}$. Here $S_j$ is shorthand for $S_{w_j}$, where $w_j$ is the $j$th boundary white vertex. To prove $v_j \in L_j$ it is equivalent to show that $v_j$ is orthogonal to $v_i'$ for each $i \in \tilde{I}_j \setminus \{j\}$.  This fact is part of the characterization of the inverse of the right twist (also known as the left twist) provided by Muller and Speyer \cite{MulSpe}.

To extend to a configuration with the desired properties, each internal $v_w$ is determined up to scale since $L_w$ is a line.  Fixing a nonzero $v_w$ for each $w$, we get by Proposition \ref{prop:circuit} that the associated relations $R_b$ are also determined up to scale.  In short, the whole configuration is determined up to gauge at internal vertices, giving us the uniqueness.  Also by Proposition \ref{prop:circuit}, the relations have nonzero coefficients which gives us the circuit condition.

It remains to show that the vectors and relations $(\v, \Rvec)$ as above comprise a valid configuration on $G$.  The only property not clear at this point is that the Kasteleyn matrix $K$ is full rank.  As already mentioned, all coefficients $K_{bw}$ with $\overline{bw} \in E$ are nonzero.  By the general theory, there is a unique almost perfect matching of $B$ with $W \setminus I_1$ (one reference is \cite[Proposition 5.13]{MulSpe} and we also describe a construction of this matching later on).  Therefore the polynomial $\Delta_{W \setminus I_1}(K)$ of the coefficients is in fact a monomial and hence nonzero.
\end{proof}

We now have our definition of the reconstruction map $\Psi: T_G \to \mathcal{C}_G$, namely it maps $A$ to the configuration given by Proposition \ref{prop:existence}.  Clearly $\Phi \circ \Psi$ is the identity.  In plainer terms we have existence of an extension of generic $A \in \Pi_{\mathcal{M}}$ to a full configuration.  In principle, there could be other extensions with $v_w \notin L_w$ for some $w$, a possibility we rule out in the next subsection.

\begin{example}
Consider the plabic graph $G$ in Figure \ref{fig:strand}.  As discussed in Example \ref{ex:gr36}, $G$ corresponds to the uniform matroid in $Gr_{3,6}$, and it follows that $I_j = \{j,j+1,j+2\}$ with indices modulo $6$.  Given $A = [v_1 \cdots v_6]$ then, $H_i = \langle v_{i+1}, v_{i+2} \rangle$.  The unique internal white vertex $w$ has $S_w = \{3,6\}$, so 
\begin{displaymath}
L_w = H_3 \cap H_6 = \langle v_4,v_5 \rangle \cap \langle v_1, v_2 \rangle.  
\end{displaymath}
Hence our general recipe reproduces the result argued in Example \ref{ex:gr36}.
\end{example}

\subsection{Uniqueness} \label{sec:unique}
Fix $A \in T_G$.  We now know $\Phi(\Psi(A)) = A$.  On the other hand, suppose $(\v,\Rvec) \in \mathcal{C}_G$ and that $\Phi(\v,\Rvec) = A$.  We want to show $(\v,\Rvec) = \Psi(A)$ in order to establish that preimages are unique.  In light of Proposition \ref{prop:existence}, it is sufficient to show $v_w \in L_w$ for all internal white vertices $w$.  The proof is in a sense recursive, utilizing a certain acyclic orientation on $G$.

A \emph{perfect orientation} on $G$ is an orientation with the property that each internal white vertex has a unique incoming edge and each (internal) black vertex has a unique outgoing edge.  Given such an orientation, the set of edges oriented from black to white always gives an almost perfect matching.  We focus on one particular perfect orientation which we denote $\mathcal{O}$ and which is defined as follows.  Each edge of $G$ is part of two zigzags that traverse it in opposite directions.  Declare each edge to be oriented in the direction of its smaller numbered zigzag.

Let $\pi$ be the almost perfect matching associated with $\mathcal{O}$.  More directly, an edge is in $\pi$ if and only if the smaller numbered zigzag through the edge traverses it from black to white.  It is easy to see that $\pi$ is among the extremal matchings defined by Muller and Speyer \cite{MulSpe} in terms of downstream / upstream wedges.  Specifically, $\pi$ is the set of edges $e$ for which the face of $G$ containing the boundary segment from $n$ to $1$ lies in the upstream wedge of $\e$.  We stick with our characterization of $\mathcal{O}$ and $\pi$, but make use of some previously established combinatorial properties.

\begin{proposition}[{\cite[Theorem 5.3 and Corollary B.7]{MulSpe}}] \label{prop:orientation}
The orientation $\mathcal{O}$ on $G$ defined above has the following properties:
\begin{enumerate}
\item It is a perfect orientation.  
\item The corresponding matching $\pi$ uses precisely the boundary vertices $\{1,\ldots, n\} \setminus I_1$.
\item The matching uses exactly $m-1$ edges from each internal $2m$-gon face.
\item The orientation is acyclic.
\end{enumerate}
\end{proposition}

\begin{proof}
The first three parts amount to a special case of \cite[Theorem 5.3]{MulSpe}.  The last one follows quickly from the cited corollary, which states that $\pi$ is the unique almost perfect matching using its set of boundary vertices.  Indeed, suppose for the sake of contradiction that the orientation had an oriented cycle.  Half of the edges of the cycle, namely those going from black to white, appear in $\pi$.  Another matching is obtained by taking out all of these edges and including the other half of the edges of the cycle.  The result is another almost perfect matching using the same boundary vertices, a contradiction.
\end{proof}

\begin{corollary} \label{cor:orientation}
Suppose $(\v, \Rvec) \in \mathcal{C}_G$ and that $\{v_j : j \in I_1\}$ is a basis for $V$.  Then $K_{bw} \neq 0$ for each edge $\overline{bw}$ in the matching $\pi$.
\end{corollary}

\begin{proof}
By Corollary \ref{cor:KToA}, we know $\Delta_{W \setminus I_1}(K) \neq 0$.  As mentioned in the proof of Proposition \ref{prop:orientation}, $\pi$ is the unique matching of $W \setminus I_1$ with $B$.  As such the determinant equals (up to sign) the product of the weights $K_{bw}$ of the edges of the matching.  Therefore each such weight must be nonzero.
\end{proof}

\begin{proposition} \label{prop:hyperplane}
Suppose $(\v, \Rvec) \in \mathcal{C}_G$ and that $\{v_j : j \in I_1\}$ is a basis for $V$.  Recall $H_1$ is the span of the vectors $v_j$ for $j \in I_1 \setminus \{1\}$.  If $w$ is a white vertex and there is no oriented (relative to $\mathcal{O}$) path from boundary vertex $1$ to $w$ then $v_w \in H_1$.
\end{proposition}

\begin{proof}
Let $b$ be the black vertex such that $\overline{bw}$ is the unique edge incident to and directed towards $w$.  By Corollary \ref{cor:orientation}, $K_{bw} \neq 0$.  As $\sum_{w'} K_{bw'}v_{w'} = 0$, we have that $v_w$ lies in the span of the $v_{w'}$ for $w'$ the other neighbors of $b$.  Note that all edges $\overline{w'b}$ with $w'\neq w$ are oriented towards $b$ so there is a length $2$ path from each $w'$ to $w$.  We can apply the same argument recursively to each $w'$.  Since the orientation is acyclic the end result is that $v_w$ lies in the span of those $v_j$ with $j$ a source (i.e. $j \in I_1$) for which there exists a path from $j$ to $w$.  By assumption there is no such path from $1$ so $v_w \in H_1$ as desired.
\end{proof}

\begin{lemma}  \label{lem:zigzag}
Let $w$ be any white vertex, internal or boundary.  If $w$ lies strictly left of the zigzag starting at $1$ then there is no oriented path from $1$ to $w$ in the orientation $\mathcal{O}$.
\end{lemma}

\begin{proof}
First note that every edge of zigzag $1$ is oriented in the direction of zigzag $1$.  In other words, zigzag $1$ is an oriented path.  We claim there is no oriented edge starting weakly right of the zigzag and ending strictly left of the zigzag.  The proof is case by case depending on the start vertex of the edge.  If the edge starts strictly right of the zigzag then it must end weakly right of the zigzag (otherwise it would cross it and break planarity).  If the edge starts on the zigzag at a black vertex $b$ then it must be the unique edge oriented away from $b$.  We know that this edge is part of the zigzag so it ends on the zigzag.  Lastly, suppose our given edge starts at a white vertex $w$ on the zigzag.  The zigzag locally looks like $b,w,b'$ where $b'$ is reached by turning maximally left at $w$.  As such, all the edges incident to $w$ lie weakly to the right of the zigzag.
\end{proof}

Combining Proposition \ref{prop:hyperplane} and Lemma \ref{lem:zigzag}, we have that $v_w$ lies on $H_1$ if $w$ is strictly left of the zigzag starting at $1$.  By cyclic symmetry the statement from the previous sentence holds with $1$ replaced by any start vertex $j$ (note that to give a direct proof one would use a different perfect orientation for each $j$).  We are now ready to prove the uniqueness result.

\begin{proposition} \label{prop:uniqueness}
Let $(\v,\Rvec) \in \mathcal{C}_G$ and suppose $A = \Phi(\v) \in T_G$.  Then $(\v,\Rvec) = \Psi(A)$.
\end{proposition}

\begin{proof}
Suppose $A = \Phi(\v) \in T_G$.  Then $A$ is in the open positroid variety so the set of columns of $A$ corresponding to $I_j$ is a basis of $V$ for all $j$.  Fix $w \in W$ internal.  For each $j \in S_w$ we have the assumptions of Proposition \ref{prop:hyperplane} (with $1$ replaced by $j$), so $v_w \in H_j$.  Therefore $v_w \in L_w$.  By Proposition \ref{prop:existence}, $(\v,\Rvec)$ is in fact the same as $\Psi(A)$.  
\end{proof}

\begin{proof}[Proof of Theorem \ref{thm:plabic}]
Part 1 was proven in Proposition \ref{prop:target}.  For part 2, we showed existence of an extension of $A$ to $\psi(A)$ in Proposition \ref{prop:existence} and uniqueness of this extension in Proposition \ref{prop:uniqueness}.

We have assumed at various points that $G$ has no isolated boundary vertex and no boundary vertex attached to a degree $1$ vertex.  We briefly describe modifications needed to handle these cases.  First suppose $j$ is an isolated boundary vertex of $G$.  One can consider the strand of $j$ to be a simple arc disjoint from $G$ starting at a point clockwise from $j$ and ending at a point counterclockwise from $j$.  All definitions and proofs go through.

Now suppose $j$ is attached to a degree $1$ vertex $b$.  The strand for $j$ loops around $b$ and self-intersects before returning to the boundary, causing a few problems including in the definition of the reconstruction map.  It is consistent to have all other $S_F,S_w, S_{b'}$ omit $j$, but there is no clear definition for $S_j$ and $S_b$.  That said, $j$ is part of every almost perfect matching so it is part of no basis of $\mathcal{M}$.  Hence any $A \in \Pi_{\mathcal{M}}$ has $j$th column $v_j = 0$.  As such, we can define $\Psi(A)$ to have $v_j = 0$ and $R_b = 1v_j$, and define all other vectors and relations in a manner independent of the index $j$.  Elsewhere, the orientation we define is ambiguous regarding how to orient the edge $\overline{jb}$.  In fact it must be oriented towards $j$ to get a perfect orientation, and this choice does not cause any other issues.
\end{proof}

\section{Connection to the boundary measurement map} \label{sec:measurement}
The boundary measurement map and boundary restriction maps are both functions landing in $\Pi_{\mathcal{M}}$.  The input to the former is given by a collection of nonzero edge weights on $G$.  The input to the latter is a vector-relation configuration $\v$ on $G$.  However, we have seen (paragraph after Lemma \ref{lem:rowspace}) that $\v$ is determined up to isomorphism by the matrix $K$ whose potentially nonzero entries are in bijection with edges of $G$.  We will see that applying the boundary measurement map to a set of edge weights is the same as applying the boundary restriction map to $K$ with entries equal to the edge weights multiplied by certain signs.  This fact explains why the formulas \eqref{eq:measurement} and \eqref{eq:restriction} take the same form and only differ in the signs of the individual terms.  To rectify these equations, we introduce a version of Kasteleyn signs for plabic graphs.

\subsection{Kasteleyn signs for plabic graphs}
Let $G$ be a plabic graph as above.  By a \emph{face} of $G$, we mean a connected component of the complement of the graph in the disk in which it is embedded.  A face is \emph{internal} if its boundary is a closed walk in $G$ and \emph{external} otherwise.  The boundary of an external face is an interlacing collection of walks in $G$ and segments of the boundary of the disk.  There is a unique external face including the boundary segment of the disk from $n$ to $1$.  We call this face the \emph{infinite face} and all other faces \emph{finite faces}.  See Figure \ref{fig:external} for an example of this terminology.

\begin{figure}
\centering
\includegraphics[height=2.5in]{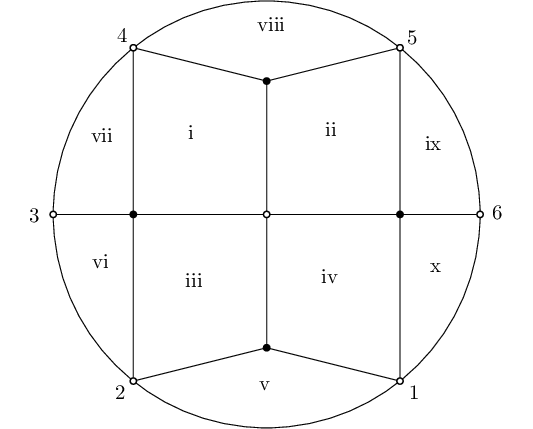}
\caption{Example illustrating the terminology for different types of faces of a plabic graph.  Here faces i through iv are internal and faces v through x are external with face x being the infinite face.}
\label{fig:external}
\end{figure}  

Let $\epsilon_{bw} = \pm 1$ for each edge $\overline{bw}$ of $G$.  Say these constitute a choice of \emph{Kasteleyn signs} if
\begin{itemize}
\item the product of the signs along the boundary of each internal $2m$-gon face is $(-1)^{m-1}$ and
\item the product of the signs along the walk(s) on the boundary of each finite external face is $(-1)^{m+a-1}$, where $a$ is the number of said walk(s) and $2m$ is the total number of edges along them.
\end{itemize}
Several notes are in order.  Each walk on the boundary of an external face starts and ends at a boundary vertex and hence has even length as we assume boundary vertices are white.  No assumption is made about the product of the signs around the infinite face.  Lastly, the most common case is that $G$ is connected.  In that event each finite external face is cut out by a single ($a=1$) path from $i$ to $i+1$ and the second condition becomes that the product of signs along this path is $(-1)^m$ where $2m$ is its length.

\begin{remark}
One reference for Kasteleyn signs on plabic graphs is a short note of Speyer \cite{Spe}.  He does not directly identify the conditions above.  Instead he defines the signs implicitly so that a certain result (along the lines of our Proposition \ref{prop:plabicSigns}) holds and then proves that such signs exist with a topological argument.
\end{remark}

\begin{proposition} \label{prop:signs}
A choice of Kasteleyn signs on $G$ exists.
\end{proposition}

\begin{proof}
Extend $G$ to a new graph $\tilde{G}$ by adding a single black vertex $b_{\infty}$ and $n$ edges connecting $b_{\infty}$ to the boundary vertices $1, \ldots, n$.  Then $\tilde{G}$ can be embedded in a sphere and the faces of $\tilde{G}$ in the standard sense biject in a natural way with the faces of $G$ as defined above.  Consider the faces of $\tilde{G}$ finite or infinite as dictated by this bijection.  By ordinary Kasteleyn theory signs can be chosen on the edges of $\tilde{G}$ so that each finite $2m$-gon face has a product of signs equal to $(-1)^{m-1}$.  This property is preserved by gauge transformations wherein all signs adjacent to a given vertex are flipped.  Applying gauge as needed at vertices $1$ through $n$ we can assume that all signs of edges adjacent to $b_{\infty}$ are positive.  Restricting the signs to the subgraph $G$ yields precisely the right properties.
\end{proof}

\subsection{The translation}
We are ready to precisely state the procedure that translates between the boundary restriction and boundary measurement maps.

\begin{proposition} \label{prop:plabicSigns}
Let $\v \in \mathcal{C}_G$ and suppose that all coefficients $K_{bw}$ for $\overline{bw} \in E$ are nonzero.  Define $\wt:E \to \mathbb{C}^*$ by $\wt(\overline{bw}) = \epsilon_{bw}K_{bw}$ for a fixed choice of Kasteleyn signs.  Then $\Phi(\v)$ equals the output of the boundary measurement map applied to this weight function.
\end{proposition}

\begin{proof}
Substitute $K_{bw} = \epsilon_{bw}\wt(\overline{bw})$ into \eqref{eq:restriction}.  Let $\epsilon_J = \pm 1$ as per the sign in the front of the summation, which as previously mentioned is based on the parity of $(j_1-1) + \ldots + (j_k-k)$.  So given $J$ and a matching $\pi$ of $B$ with $W \setminus J$, the sign of the corresponding term is
\begin{equation} \label{eq:sign}
\epsilon_J \sgn(\pi)\prod_{\overline{bw} \in \pi} \epsilon_{bw}.
\end{equation}
To match \eqref{eq:measurement} we need to show all these signs are equal (it is okay if they are all negative as the Pl\"ucker coordinates are only defined up to a common multiple) as $J$ and $\pi$ vary.  For two matchings with the same $J$ this property is standard for Kasteleyn theory.  One possible reference is \cite[Theorem 2]{Ken09}, and in fact we will follow the same outline in our proof of the general case.

We will compare the signs from \eqref{eq:sign} corresponding to the pair $J_1, \pi_1$ and the pair $J_2, \pi_2$.  The disjoint union of the edges of $\pi_1$ and $\pi_2$ gives a multigraph for which each internal vertex has degree $2$ and each boundary vertex has degree $0$, $1$, or $2$.  Each component (not counting isolated boundary vertices) is a doubled edge, a cycle, or a path starting and ending at the boundary.  Each path and cycle alternates between edges of $\pi_1$ and $\pi_2$.  Starting from $\pi_1$ one can \emph{flip} along such a component by switching to the other half of the edges to obtain a matching with greater overlap with $\pi_2$.  By induction it suffices to consider the case when $\pi_1$ and $\pi_2$ are related by a single flip.  As already mentioned, the case of flipping a cycle in the graph is classical, so we focus on the path case.

Suppose $\pi_1$ and $\pi_2$ are related by a flip of a path from $i$ to $j$ with $i < j$.  Without loss of generality, $\pi_2$ contains the edge of the path incident to $i$.  It follows that $J_2 = J_1 \setminus \{i\} \cup \{j\}$.  Therefore
\begin{equation} \label{eq:kast1}
\frac{\epsilon_{J_2}}{\epsilon_{J_1}} = (-1)^{j-i}.
\end{equation}

We next consider the signs of the matchings.  To make comparison easier, pass to the graph $\tilde{G}$ from the proof of Proposition \ref{prop:signs}.  Extend the matchings to $\tilde{\pi}_1 = \pi_1 \cup \{\overline{ib_{\infty}}\}$ and $\tilde{\pi}_2 = \pi_2 \cup \{\overline{jb_{\infty}}\}$.  Both are matchings of $B \cup \{b_{\infty}\}$ with $W \setminus (J_1 \cap J_2)$.  They are related by a flip in $\tilde{G}$ about a cycle consisting of the original path from $i$ to $j$ along with the edges from $i$ and $j$ to $b_{\infty}$.  It follows that $\sgn(\tilde{\pi}_2) = (-1)^{m-1}\sgn(\tilde{\pi}_1)$ where $2m$ is the number of edges of this cycle.  Now consider $\pi_1$ as an $M \times M$ permutation matrix with columns indexed by $W \setminus J_1$.  Then $\tilde{\pi}_1$ is obtained by adding a row to the end corresponding to $b_{\infty}$, adding a column in the appropriate place corresponding to $i$, and putting a $1$ where the new row and column intersect.  The columns right of the new one are indexed by $\{i+1,\ldots, N\} \setminus J_1$, so
\begin{displaymath}
\sgn(\tilde{\pi}_1) = (-1)^{|\{i+1,\ldots, N\} \setminus J_1|} \sgn(\pi_1).
\end{displaymath}
By a similar argument
\begin{displaymath}
\sgn(\tilde{\pi}_2) = (-1)^{|\{j+1,\ldots, N\} \setminus J_2|} \sgn(\pi_2).
\end{displaymath}
Putting the pieces together
\begin{equation} \label{eq:kast2}
\frac{\sgn(\pi_2)}{\sgn(\pi_1)} = (-1)^{m-1+|\{i+1,\ldots, j\} \setminus J_1|}
\end{equation}
using the fact that $J_1$ and $J_2$ agree beyond $j$.

The last consideration is the sign coming from the weights on the edges where $\pi_1$ and $\pi_2$ differ, i.e. along the path from $i$ to $j$.  As in the previous paragraph we complete this to a cycle of length $2m$ passing through $b_{\infty}$.  This addition has no effect on signs because, as in the proof of Proposition \ref{prop:signs} we take all edges adjacent to $b_{\infty}$ to have sign $+1$.  We have a cycle in $\tilde{G}$, a graph with ordinary Kasteleyn signs, so by \cite[Lemma 1]{Ken09} the signs around it come to $(-1)^{1+m+l}$ where $l$ is the number of vertices properly inside the cycle, ``inside'' referring to the component that does not include the infinite face.  By our conventions, this inside region includes the boundary vertices $i+1,\ldots, j-1$ and not the others.  Restricted to this region, the matching $\pi_1$ includes all vertices other than $\{i+1,\ldots, j-1\} \cap J_1$.  The vertices covered by $\pi_1$ come in pairs so $l$ has the same parity as $|\{i+1,\ldots, j-1\} \cap J_1|$.  Therefore
\begin{equation} \label{eq:kast3}
\frac{\prod_{\overline{bw} \in \pi_2} \epsilon_{bw}}{\prod_{\overline{bw} \in \pi_1} \epsilon_{bw}} = (-1)^{1+m + |\{i+1,\ldots, j-1\} \cap J_1|}.
\end{equation}  

Multiplying \eqref{eq:kast1}, \eqref{eq:kast2}, and \eqref{eq:kast3}, we get that the ratio of the signs of the two terms is governed by the parity of 
\begin{displaymath}
2m + j-i+ |\{i+1,\ldots, j\} \setminus J_1| + |\{i+1,\ldots, j-1\} \cap J_1|.
\end{displaymath}
which equals $2m + j-i + j-i = 2(m+j-i)$ (using that $j \notin J_1$).  This number is even so the terms have the same sign as desired.

\end{proof}

\subsection{Geometric interpretation of edge weights}
The reconstruction map allows to construct a gauge class of vector relation configurations on a plabic graph from a suitable point in the Grassmannian.  One could then fix a representative of the gauge class and a choice of Kasteleyn signs to determine the edge weights, obtaining the preimage of the original point under the boundary measurement map.  The edge weights are not unique, but in this subsection we describe how to calculate one valid set of edge weights directly.  Our method uses an efficient recursive formulation of the boundary measurement map.  We assume for this subsection that all edge weights are positive reals.  Note that the problem of recovering the edge weights was solved previously for Le-diagrams by Talaska \cite{Tal} and in general by Muller and Speyer \cite{MulSpe} (the latter also allowing for complex weights).

Assume throughout that $G$ is oriented using the perfect orientation $\mathcal{O}$ from Section \ref{sec:unique}.  The set of sources of this orientation is $I_1$.  Suppose positive real edge weights are given on $G$.  Fix a basis $\{\tilde{v}_i : i \in I_1\}$ of $\mathbb{R}^k$.  For any non-source white vertex $w$, define $\tilde{v}_w$ as follows.  Let $b$ be its unique neighbor such that $e = \overline{bw}$ is directed towards $w$.  All other neighbors $w'$ of $b$ are such that $\overline{bw'}$ is directed towards $b$.  Let
\begin{equation} \label{eq:recursion}
\tilde{v}_{w} = \frac{1}{\wt(e)} \sum_{w'} \wt(\overline{bw'})\tilde{v}_{w'}.
\end{equation}
As the orientation is acyclic, this is a sensible recursive definition.

Now let
\begin{displaymath}
\sigma_j = (-1)^{|I_1 \cap \{1,2,\ldots,j\}|-1}.
\end{displaymath}
and $v_j = \sigma_j \tilde{v}_j$ for $j=1,\ldots, n$.  Let $A = [v_1 \cdots v_n]$.

\begin{proposition}
The point $A \in Gr_{k,n}$ agrees with the output of the boundary measurement map applied to the weighted graph.
\end{proposition}
\begin{proof}[Proof sketch]
There is a solution to the defining recurrence of the $\tilde{v}_w$ expressing each such vector as
\begin{displaymath}
\tilde{v}_w = \sum_{i \in I_1} M_{iw} v_i
\end{displaymath}
where $M_{iw}$ is the sum of weights of paths from $i$ to $w$ with respect to a certain notion of weight.  The matrix $[\tilde{v}_1 \cdots \tilde{v}_n]$ nearly agrees with Postnikov's original definition (which is made simpler in this case since our orientation is acyclic) of the boundary measurement map \cite{Pos06}.  The discrepancy is that Postnikov multiplies each entry by a sign ($(-1)^s$ in his notation where $s$ depends on $i$ and $j$).  Our approach of multiplying column $j$ by $\sigma_j$ produces the same point of the Grassmannian.  We choose to omit the details of Postnikov's original construction and of the equivalence with our choice of signs.
\end{proof}

Now, the $\tilde{v}_w$ are the vectors of a vector-relation configuration with the \eqref{eq:recursion} being the relations.  Using acyclicity again, we can apply gauges at internal vertices so that the sum of the weights of incoming edges to each internal vertex equals $1$.  In the notation above, $e$ is the unique incoming edge to $w$ so if $w$ is internal then $\wt(e) = 1$ and the coefficients in \eqref{eq:recursion} sum to $1$.  So $\tilde{v}_w$ is a convex combination of the $\tilde{v}_w'$ in this case.  This is our motivation to choose this representative of the gauge class.

\begin{example} \label{ex:convex}
Figure \ref{fig:convex} shows a plabic graph $G$ for $Gr_{3,6}$.  As $I_1 = \{1,2,3\}$ only $\sigma_2 = -1$ so we can suppress the $\sim$'s in the $\tilde{v}_j$ except for $j=2$.  Let $u$ be the vector at the interior white vertex.  Let $v_1, \tilde{v}_2, v_3$ be any basis of $\mathbb{R}^3$.  Following the arrows we construct
\begin{align*}
u &= a_1v_1 + a_2\tilde{v}_2 \\
v_4 &= \frac{1}{b_0}(b_1\tilde{v}_2 + b_2v_3 + b_3u) \\
v_5 &= \frac{1}{c_0}(c_1v_4 + c_2u) \\
v_6 &= \frac{1}{d_0}(d_1v_1 + d_2v_5 + d_3u)
\end{align*}
The output of the boundary measurement map is $[v_1 v_2 \cdots v_6]$ where $v_2 = -\tilde{v}_2$.  

\begin{figure}
\centering
\includegraphics[height=2.5in]{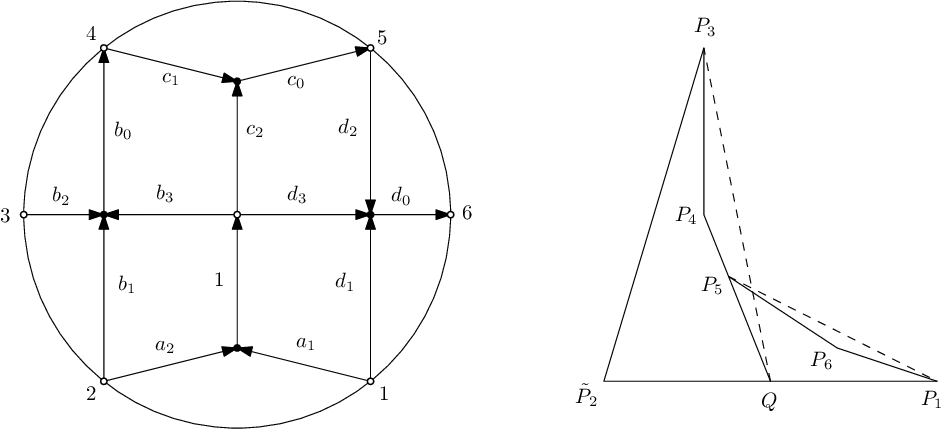}
\caption{A plabic graph with an acyclic perfect orientation (left) and a configuration that results from the associated sequence of convex combinations (right)}
\label{fig:convex}
\end{figure}

If we allow gauge at boundary vertices, which corresponds to modding out by the torus action on the Grassmannian, we can additionally arrange $b_0=c_0=d_0=1$.  Each vector is constructed as a convex combination of its predecessors with coefficients given by the edge weights.  It is easier to draw the picture in the projective plane replacing the vectors $v_1, \tilde{v}_2, \ldots v_6, u$ with points $P_1, \tilde{P}_2, \ldots, P_6, Q$ (see the right of Figure \ref{fig:convex}).  The points are constructed in the same way: start with $P_1,\tilde{P}_2,P_3$, pick $Q$ on segment $\overline{P_1\tilde{P}_2}$, pick $P_4$ in triangle $\triangle \tilde{P}_2P_3Q$, pick $P_5$ on segment $\overline{P_4Q}$, and pick $P_6$ in $\triangle P_1P_5Q$.
\end{example}

The above all concerns the forward boundary measurement map.  Going in the other direction, let $A = [v_1 \cdots v_n] \in GR_{k,n}$ be totally positive, meaning that all its Plucker coordinates are positive.  Consider the problem of reconstructing the positive edge weights.  Let $\tilde{A} = [\tilde{v}_1 \cdots \tilde{v}_n] \in \Pi_{\mathcal{M}}^{\circ}$ where $\tilde{v}_j = \sigma_j v_j$ for $\sigma_j$ as before.  Finally, let $\vtilde = \psi(\tilde{A})$ where the gauge class is chosen recursively so that each $\tilde{v}_w$ for $w$ internal is a convex combination of the $\tilde{v}_{w'}$ for $w'$ two steps upstream from $w$.

\begin{proposition}
Define edge weights on $G$ as follows.  Given $b \in B$, let $w$ be the target of the unique edge $e$ directed away from $b$.  Below let $w'$ range over the other neighbors of $b$.
\begin{itemize}
\item Suppose $w$ is internal.  Then put $\wt(e) = 1$ and let the $\wt(\overline{bw'})$ be the barycentric coordinates of $\tilde{v}_w$ with respect to the $\tilde{v}_{w'}$.
\item Suppose $w$ is on the boundary.  Then put $\wt(e) = \lambda$ where $\lambda$ is chosen so that $\lambda \tilde{v}_{w}$ is a convex combination of the $\tilde{v}_{w'}$.  Let the $\wt(\overline{bw'})$ be the associated barycentric coordinates of $\lambda \tilde{v}_{w}$.
\end{itemize}
Then, this edge weighting is a representative of the inverse of the boundary measurement map applied to $A$.
\end{proposition}

The proof is immediate as we are just undoing the boundary measurement map as described in this section.  The recipe for the edge weights is purely geometric: form and intersect some hyperplanes (as in the definition of $\Psi$), do some projections, and take some barycentric coordinates.

\begin{example}
Continue with $G$ as in Figure \ref{fig:convex} and Example \ref{ex:convex}.  Let $A = [v_1 \cdots v_6] \in Gr_{3,6}$ with all minors positive.  Consider the problem of determining the edge weights $a_1$ and $a_2$.

Let $\tilde{v}_2 = -v_2$.  As discussed in previous examples, the internal vector $u$ satisfies
\begin{displaymath}
u \in \langle v_1, v_2 \rangle \cap \langle v_4, v_5 \rangle.
\end{displaymath}
Positivity will ensure that $u$ can be scaled to be a convex combination of $v_1$ and $\tilde{v}_2$.  Then, $a_1,a_2$ are the corresponding barycentric coordinates.

Alternately, a direct formula for $a_1,a_2$ can be derived as follows.  There is a determinantal identity
\begin{displaymath}
|v_2v_4v_5|v_1 - |v_1v_4v_5|v_2 = -|v_1v_2v_5|v_4 + |v_1v_2v_4|v_5
\end{displaymath}
giving a vector on the desired line $\langle v_1, v_2 \rangle \cap \langle v_4, v_5 \rangle$.  We want a convex combination
\begin{displaymath}
a_1v_1 + a_2\tilde{v}_2 = a_1v_1 - a_2v_2
\end{displaymath}
so we scale down to get
\begin{displaymath}
a_1 = \frac{|v_2v_4v_5|}{|v_1v_4v_5| + |v_2v_4v_5|}, \quad \quad a_2 = \frac{|v_1v_4v_5|}{|v_1v_4v_5| + |v_2v_4v_5|}.
\end{displaymath}
These values are positive when the minors of $A$ are positive as expected.
\end{example}

\subsection{The circuit condition}
We conclude this section with a discussion of configurations $\v \in \mathcal{C}_G$ satisfying the circuit condition at each black vertex.

Define $\mathcal{C}^{\circ}_G \subseteq \mathcal{C}_G$ to be the set of configurations (up to gauge and $GL_k$ action) with all coefficients of all relations nonzero.  The statement of Proposition \ref{prop:plabicSigns} can be summarized by saying there is an identification of $\mathcal{C}^{\circ}_G$ with the set of gauge classes of nonzero weights on $G$ such that $\Phi|_{\mathcal{C}^{\circ}_G}$ agrees with the boundary measurement map.  We now give a slightly stronger formulation of Theorem \ref{thm:plabic}.

\begin{theorem} \label{thm:main}
Let $G$ be a reduced plabic graph.
\begin{enumerate}
\item The image of $\Phi$ is contained in $\Pi_{\mathcal{M}}$.
\item The restriction of $\Phi$ to $\mathcal{C}_G^{\circ}$ is an isomorphism with its image $T_G$.
\item $\mathcal{C}_G^{\circ} = \Phi^{-1}(T_G)$.
\end{enumerate}
\end{theorem}

\begin{proof}
The first part is a restatement of the first part of Theorem \ref{thm:plabic}.  From the second part of Theorem \ref{thm:plabic} we have that each $A \in T_G$ has a unique preimage under $\Phi$.  More precisely, we know by Proposition \ref{prop:existence} that the preimage $\Psi(A)$ has nonzero coefficients in all its relations, i,e, $\Psi(A) \in \mathcal{C}_G^{\circ}$.  All that remains for both the second and third parts of the current theorem is to show that $\Phi (\mathcal{C}_G^{\circ}) \subseteq T_G$.  As $\Phi$ has the same image as the boundary measurement map, the previous follows from Theorem \ref{thm:mulspe}.
\end{proof}

\begin{proposition}
Let $(\v,\Rvec) \in \mathcal{C}_G$.  Then $(\v, \Rvec) \in \mathcal{C}_G^{\circ}$ if and only if the $v_w$ neighboring each $b \in B$ are a circuit.
\end{proposition}

\begin{proof}
Any non-trivial linear relation on the elements of a circuit must have nonzero coefficients, so the if direction is easy.  On the other hand, if $\v \in \mathcal{C}_G^{\circ}$, then by Theorem \ref{thm:main} we have $\Phi(v) \in T_G$ so $\v = \Psi(\Phi(\v))$ satisfies the circuit condition by Proposition \ref{prop:existence}.
\end{proof}

\section{Structure of the space $\mathcal{C}_G$} \label{sec:structure}
Let $G$ be a plabic graph with all of the conventions and notation of Section \ref{sec:plabic}. We have defined $\mathcal{C}_G$ as the set of vector-relation configurations on $G$ modulo gauge transformations at internal vertices and the action of $GL_k(\mathbb{C})$.  We now consider the algebraic-geometric structure both of $\mathcal{C}_G$ and of the function $\Phi: \mathcal{C}_G \to \Pi_{\mathcal{M}}$.  We know $T_G \subseteq \Pi_{\mathcal{M}}$ is dense, and by Theorem \ref{thm:main} the difference of these sets is mapped to by $\mathcal{C}_G \setminus \mathcal{C}_G^{\circ}$.  The main result of this Section is that $\mathcal{C}_G$ is a smooth algebraic variety.  Unfortunately, the map $\Phi$ is not always surjective, but we will see that it does resolve singularities of the positroid variety in some cases.

First, consider a configuration $\v \in \mathcal{C}_G$.  We will describe explicitly a bijection between a neighborhood of $\v$ and an open subset of an affine variety.  Since the boundary vectors $v_1,\ldots, v_n$ span $V$, there must be a basis $\{v_j : j \in I\}$ among them.  Acting by $GL_k$ we can arrange for this to equal the standard basis in order.  Next, each internal vector $v_w$ is nonzero so we can pick one of its nonzero entries and apply a gauge so that the entry equals $1$.  Finally, by Lemma \ref{lem:determinant} we know that 
\begin{displaymath}
\Delta_{W \setminus I}(K) \neq 0.
\end{displaymath}
It follows that there is an almost perfect matching of $B$ with $W \setminus I$ with all $K_{bw}$ along the matching nonzero.  Apply gauge at the black vertices to scale all these $K_{bw}$ to $1$.

We have exhausted the allowable operations, so the collection of remaining variables gives a well-defined map to affine space.  Specifically, the coordinates are the entries of the boundary vectors $v_j$ with $j \notin I$, each entry of each internal vector $v_w$ except the one scaled to $1$, and all the $K_{bw}$ for edges $\overline{bw}$ not in the matching.  The map to affine space is injective and it is easy to describe the image.  For each $b\in B$ the vector relation $\sum_w K_{bw}v_w = 0$ amounts to $k$ quadratic relations in the variables.  The only other condition is that the matrix $K$ has full rank.  Restricting the chart a bit, we can replace the full-rank condition with the single inequality $\Delta_{W \setminus I}(K) \neq 0$ which as already mentioned holds for $\v$.

\begin{example} \label{ex:gr24}
The graph $G$ in Figure \ref{fig:gr24} is one of the standard plabic graphs for the open cell in $Gr_{2,4}$.  Suppose a point $\v \in \mathcal{C}_G$ is given such that $\{v_1,v_3\}$ is a basis of $\mathbb{C}^2$, $v_2$ appears non-trivially in the relation on $v_2,v_3,v_4$, and $v_4$ appears non-trivially in the relation on $v_1,v_2,v_4$.  The normalization described above produces the configuration in the figure where the edge variables indicate the coefficients of the relations.  From the vector relations $v_2 + av_3 + bv_4 = 0$ and $cv_1 + dv_2 + v_4 = 0$ we obtain the system
\begin{align*}
x_2 + bx_4 &= 0 \\
y_2 + a + by_4 &= 0 \\
c + dx_2 + x_4 &= 0 \\
dy_2 + y_4 &= 0
\end{align*}
The Kasteleyn matrix is
\begin{displaymath}
\left[\begin{array}{cccc} 0 & 1 & a & b \\ c & d & 0 & 1\\ \end{array}\right].
\end{displaymath}
Taking the columns not in our basis we want $\Delta_{24}(K) = 1-bd \neq 0$.  The image of the chart is defined in affine space $\mathbb{C}^8$ by the four equations above and this one inequality.

\begin{figure}
\centering
\includegraphics[height=2in]{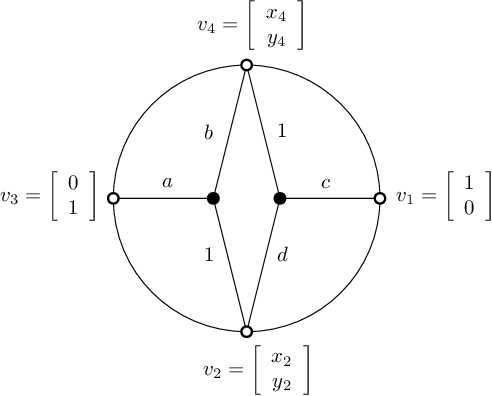}
\caption{One chart on $\mathcal{C}_G$ with $G$ a plabic graph corresponding to $Gr_{2,4}$.}
\label{fig:gr24}
\end{figure}

In fact, a more efficient chart is obtained by taking just $a,b,c,d$ as coordinates.  The other variables can be reconstructed as
\begin{align*}
&x_2 = \frac{bc}{1-bd} \quad \quad \ \  x_4 = -\frac{c}{1-bd} \\
&y_2 = -\frac{a}{1-bd} \quad \quad y_4 = \frac{ad}{1-bd}
\end{align*}
and as before we assume $1-bd \neq 0$.  The image of the chart is an open subset of $\mathbb{C}^4$, so in particular it is smooth.
\end{example}

\begin{theorem} \label{thm:smooth}
The space $\mathcal{C}_G$ of configurations on any reduced plabic graph $G$ is smooth.
\end{theorem}

The images of the charts defined above have lots of defining equations which make analysis somewhat difficult.  Generalizing Example \ref{ex:gr24}, we introduce more intricate charts which have the advantage of landing in open subsets of affine space.  The atlas that results is indexed by certain subgraphs of $G$.

Say a subgraph $F = (B \cup W, E')$ of $G = (B \cup W, E)$ is a \emph{system} in $G$ if
\begin{itemize}
\item $F$ is a forest,
\item each component of $F$ includes exactly one boundary vertex of $G$, and
\item each component of $F$ either contains exactly one edge or has the property that all of its black vertices have degree $2$.
\end{itemize}
We choose the name because $F$ has the appearance of a system of rivers connecting various points on an island to the surrounding ocean.

\begin{proposition} \label{prop:systemGauge}
Let $F = (B \cup W, E')$ be a system in $G$.  Suppose a configuration $\v \in \mathcal{C}_G$ has the property that $K_{bw} \neq 0$ for all $\overline{bw} \in E'$.  Then there is a unique representative of the gauge class of $\v$ so that $K_{bw} = 1$ for all $\overline{bw} \in E'$.  
\end{proposition}

\begin{proof}
There is a unique simple path in $F$ from each internal vertex to the boundary.  Define a partial order on the set of internal vertices via $a \preceq a'$ if $a$ lies on the path from $a'$ to the boundary.  Go through the internal vertices in a manner consistent with this order.  At each $a$ apply gauge to set equal to $1$ the coefficient at the first edge on the path from $a$ to the boundary.  Each $K_{bw}$ will be set to $1$ after the gauge at whichever of $b$ or $w$ is larger in the order, and it will remain unchanged thereafter.  It is easy to see that all choices for this gauge were forced, so the outcome is unique.
\end{proof}

We now have a rational map $\phi_F: \mathcal{C}_G \to \mathbb{C}^{|E \setminus E'|}$.  The map takes as input a configuration $\v$, performs the gauge described in Proposition \ref{prop:systemGauge}, and outputs the remaining coefficients $K_{bw}$ for $\overline{bw} \notin E'$.  On the other hand, given a point in $c \in \mathbb{C}^{|E \setminus E'|}$ we can construct a matrix $K$ by setting 
\begin{displaymath}
K_{bw} = \begin{cases}
1, & \textrm{if } \overline{bw} \in E' \\
c_{bw}, & \textrm{if } \overline{bw} \in E \setminus E' \\
0, & \textrm{otherwise} 
\end{cases}
\end{displaymath}
As explained in the paragraph following Lemma \ref{lem:rowspace}, an element of $\mathcal{C}_G$ is determined by its Kasteleyn matrix so we can recover $\v$.  Therefore $\phi_F$ is injective.

\begin{proposition} \label{prop:chartDense}
The image of $\phi_F$ is dense in $\mathbb{C}^{|E \setminus E'|}$.
\end{proposition}

\begin{proof}
Given $c \in \mathbb{C}^{|E \setminus E'|}$ we can construct $K$ as above and then use Lemma \ref{lem:rowspace} to build a configuration $\v$ in $\mathbb{C}^N / \row(K)$.  The only difficulty would be if $\v$ violated one of the conditions in the definition of $\mathcal{C}_G$, each of which is easily seen to be stated in terms of an inequality.  Specifically we need
\begin{itemize}
\item $K$ is full rank, i.e. some $\Delta_S(K) \neq 0$,
\item the boundary vectors $v_1,\ldots, v_n$ span $V$ (equivalently some subset of them is a basis), i.e. some $\Delta_{W \setminus J}(K) \neq 0$ with $J \subseteq \{1,\ldots, n\}$, and
\item for all $w$ internal $v_w$ is nonzero (equivalently $v_w$ is part of a basis with other vectors of $\v$), i.e. some $\Delta_S(K) \neq 0$ with $w \notin S$.
\end{itemize}
\end{proof}

Everything so far has only used the first two conditions of a system, namely that it is a forest with exactly one boundary vertex per component. The next result, which identifies the origin of each chart, clarifies the significance of the third condition.

\begin{proposition}
The origin $0 \in \mathbb{C}^{|E \setminus E'|}$ is in the image of $\phi_F$.  It represents a configuration where certain $k$ boundary vectors form a basis, the other boundary vectors are $0$, and each internal vector is proportional to the boundary vector in the same component.
\end{proposition} 

\begin{proof}
First note each single edge component of $F$ is balanced as it has $1$ black and $1$ white vertex.  Meanwhile, every other component has exactly one more white vertex than black.  Indeed the number of edges of the component is one less than the number of vertices (since it is a tree) and twice the number of black vertices (since each black vertex has degree $2$).  So the number of non-single edge components must equal $N-M = k$.  Let $J$ be the set of boundary vertices of these components.

Let $F_j$ be the component of $F$ containing boundary vertex $j$.  Fix $j \in J$ and $w$ a white vertex of $F_j$.  Then there is a unique matching in $F_j$ of all vertices other than $w$.  Indeed, each white vertex is paired with its neighbor on the unique path in $F_j$ from it to $w$ and each black vertex $b$ is paired with the one of its two neighbors not on the path from $b$ to $w$.  Letting $j$ vary, we get a characterization of every matching of $B$ into $W$ (using only edges of $F$):
\begin{itemize}
\item for each $j \in J$, the matching restricted to $F_j$ equals the matching described above excluding $w$ for some white vertex $w$ of $F_j$,
\item for each $j \notin J$, the matching must include the single edge of $F_j$.
\end{itemize}

Let $\pi$ be the matching where we choose to exclude the boundary vertex $j$ of $F_j$ for each $j \in J$.  Then $\pi$ is an almost perfect matching involving the white vertices $W \setminus J$.  Similarly, let $w \in W$ be any internal vertex.  Then $w$ is in $F_j$ for some $j \in J$.  Define $\pi_w$ to be the matching that excludes $w$ as well as the boundary vertices of $J$ other than $j$.  Then $\pi_w$ is a matching of $B$ with $W \setminus S$ where $S = (J \setminus \{j\}) \cup \{w\}$.

Consider the point $0 \in \mathbb{C}^{|E \setminus E'|}$ which corresponds to the zero-one matrix $K$ with $K_{bw}=1$ precisely for $\overline{bw} \in E'$.  The matchings $\pi$ and $\pi_w$ witness the conditions from Proposition \ref{prop:chartDense} for this point to be in the image of $\phi_F$.  Consider the associated configuration $\v$.  If $j \in \{1,\ldots, n \} \setminus J$ then $j$ is part of a single edge component.  The black vertex of this component gives the relation $v_j = 0$.  It follows that the $v_j$ for $j \in J$ must be a basis.  Each internal black vertex gives a relation $u+u' = 0$ among its two neighbor vectors in $F$.  By induction, all vectors in a given component of $F$ are proportional to each other.
\end{proof}

\begin{proof}[Proof of Theorem \ref{thm:smooth}] Let $\v \in \mathcal{C}_G$.  We will construct a system $F$ so that $\phi_F$ is defined at $\v$.  As such we get an identification of a neighborhood of $\v$ with an open set in affine space proving that $\mathcal{C}_G$ is smooth at $\v$.

Let $J \subseteq \{1,\ldots, n\}$ be such that $\{v_j : j \in J\}$ is a basis of $V$.  Then there is an almost perfect matching $\pi$ of $B$ with $W \setminus J$ such that $K_{bw} \neq 0$ for all $\overline{bw} \in \pi$.  We will start from the graph $F = (B \cup W, \pi)$ and add edges one at a time maintaining the properties
\begin{itemize}
\item $F$ is a forest
\item each component of $F$ includes at most one boundary vertex
\item each non-single edge component of $F$ is connected to the boundary and has all its black vertices degree $2$
\item each edge of $F$ has a nonzero coefficient $K_{bw}$
\end{itemize}
until all vertices are connected to the boundary.  The result will be a system $F$ with $\phi_F$ defined at $\v$.

Suppose we are at a stage where not all vertices of $F$ are connected to the boundary.  There are never isolated black vertices, so there must be a white vertex $w$ not connected to the boundary.  Since $v_w \neq 0$, it can be swapped into $\{v_j : j \in J\}$ so that the result is still a basis.  Suppose $j \in J$ is such that $v_j$ is the vector that got swapped out.  Then there is a matching $\pi'$ of $B$ that avoids the white vertices $(J \setminus \{j\}) \cup \{w\}$ and that has all edge variables nonzero.  The disjoint union of $\pi$ and $\pi'$ has degree $\leq 2$ at each vertex and degree $1$ at only $w$ and $j$.  As such it contains a path from $w$ to $j$ which we consider oriented in that direction.  Let $e$ be the first edge of this path whose target is connected to the boundary in $F$.  Then the source of $e$ is not connected to the boundary in $F$, so $e$ is not in $F$ and in particular $e$ is not in $\pi$.  Therefore $e$ is in $\pi'$ from which it follows that $e$ goes from black to white.  Adding $e$ to $F$ merges a single edge internal component to a boundary-connected component along a black vertex of the former.  All the properties listed above are easily verified.

As it is always possible to find an additional edge, the above process does not terminate until all vertices are connected to a boundary.  It follows at that point that $F$ is a system as desired.
\end{proof}

Combining with results from the previous section, we have that $\Phi : \mathcal{C}_G \to \Pi_{\mathcal{M}}$ maps a smooth variety to the positroid variety and restricts to an isomorphism from $\mathcal{C}_G^{\circ}$ to $T_G$. It would be interesting to characterize the image of $\mathcal{C}_G$ under $\Phi$. As suggested by the referee, a possible candidate could be the union of several Deodhar strata \cite{TW}, which are a refinement of positroid strata and are indexed by weighted networks resembling the coordinate charts associated to our systems. Although $\Phi$ is not surjective, it can resolve certain singularities of $\Pi_{\mathcal{M}}$ as the next example illustrates.

\begin{example}
Consider the plabic graph $G$ in Figure \ref{fig:singular}.  The four edges of a system $F$ have been labeled with $1$'s and the other three edges assigned coordinates $a,b,c$.  We have that $J = \{3,4\}$ are the boundary vertices of the non-single edge components of $F$ so we can take $v_3,v_4$ as a basis.  It is then possible to determine the vectors at the other three white vertices.  In these coordinates, the map $\Phi$ takes the form
\begin{displaymath}
(a,b,c) \in \mathbb{C}^3 \mapsto 
\left[ \begin{array}{cccc}
b & c & 1 & 0 \\
ab & ac & 0 & 1
\end{array} \right] \in \Pi_{\mathcal{M}} \subseteq Gr_{2,4}.
\end{displaymath}

\begin{figure}
\centering
\includegraphics[height=2in]{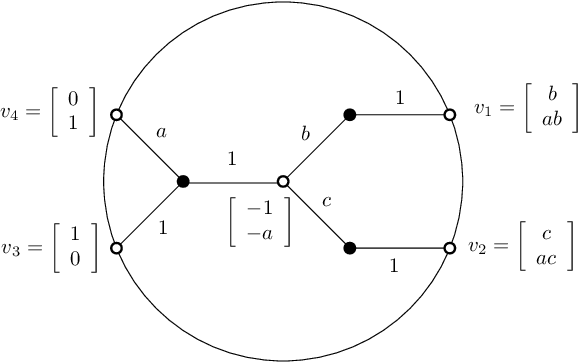}
\caption{An example of a coordinate chart on $\mathcal{C}_G$ coming from a system in $G$.}
\label{fig:singular}
\end{figure}

The target $\Pi_{\mathcal{M}}$ of $\Phi$ in this case is a Schubert variety defined in $Gr_{2,4}$ by the single equation $\Delta_{12} = 0$.  This variety has a unique singular point given in matrix form by 
\begin{displaymath}
A_{\textrm{sing}} = \left[ \begin{array}{cccc}
0 & 0 & 1 & 0 \\
0 & 0 & 0 & 1
\end{array} \right].
\end{displaymath}
If $A = [v_1 v_2 v_3 v_4] \neq A_{\textrm{sing}}$ then $v_1,v_2$ are dependent but not both zero, so they span a line.  The rightmost two black vertices in the plabic graph force the vector $u$ at the internal white vertex to lie on this line.  One can check, then, that the three relations of a configuration are always determined up to scale.  So $A$ has a unique preimage in $\mathcal{C}_G$.  

On the other hand, let $A = A_{\textrm{sing}}$.  Then $v_1=v_2=0$ and the internal vector $u \in \mathbb{C}^2$ becomes arbitrary.  As we only consider $u$ up to scale, we have a $\mathbb{P}^1$ worth of preimages.  This is the standard picture of the blowup of a variety at a point.  Alternately, we can analyze the situation in coordinates.  Restricting to the above chart we get a set $\{(a,0,0) : a \in \mathbb{C}\}$ of preimages of $A$.  The last preimage, corresponding to the point at infinity in $\mathbb{P}^1$, lies outside the chart.
\end{example}

\label{Bibliography}
\bibliographystyle{plain}
\bibliography{bibliographie}

\end{document}